\UseRawInputEncoding
\documentclass[12pt, reqno]{amsart}
\usepackage[margin=1in]{geometry}
\usepackage{amssymb,latexsym,amsmath,amscd,amsfonts}
\usepackage{latexsym}
\usepackage[mathscr]{eucal}
\usepackage{bm}
\usepackage{mathptmx}
\usepackage{amssymb}
\usepackage{amsthm}
\usepackage{dcolumn}
\usepackage[all]{xy}
\usepackage{enumitem}
\usepackage[utf8]{inputenc}

\def \qed {\hfill \vrule height6pt width 6pt depth 0pt}
\def\textmatrix#1&#2\\#3&#4\\{\bigl({#1 \atop #3}\ {#2 \atop #4}\bigr)}
\def\dispmatrix#1&#2\\#3&#4\\{\left({#1 \atop #3}\ {#2 \atop #4}\right)}

\newcommand{\beg}{\begin{equation}}
        \newcommand{\eeg}{\end{equation}}
\newcommand{\ben}{\begin{eqnarray*}}
        \newcommand{\een}{\end{eqnarray*}}

\newtheorem{theorem}{Theorem}[section]
\newtheorem{cor}[theorem]{Corollary}
\newtheorem{Lemma}[theorem]{Lemma}

\numberwithin{equation}{section} \theoremstyle{definition}
\newtheorem{definition}[theorem]{Definition}
\newtheorem{remark}[theorem]{Remark}

\newtheorem{example}[theorem]{Example}

\def\textmatrix#1&#2\\#3&#4\\{\bigl({#1 \atop #3}\ {#2 \atop #4}\bigr)}
\def\dispmatrix#1&#2\\#3&#4\\{\left({#1 \atop #3}\ {#2 \atop #4}\right)}

\begin{document}
	\title[Dilation and Birkhoff-James orthogonality]{Dilation and Birkhoff-James orthogonality}
	\author[Pal and Roy]{Sourav Pal and Saikat Roy}
	
	\address[Sourav Pal]{Mathematics Department, Indian Institute of Technology Bombay,
		Powai, Mumbai - 400076, India.} \email{sourav@math.iitb.ac.in , souravmaths@gmail.com}
	
	\address[Saikat Roy]{Mathematics Department, Indian Institute of Technology Bombay, Powai, Mumbai-400076, India.} \email{saikat@math.iitb.ac.in}		
	
	
	
	

	\begin{abstract}
	
For any two elements $x$ and $y$ in a normed space $\mathbb{X}$, $x$ is said to be Birkhoff-James orthogonal to $y$, denoted by $x\perp_B y$, if $ \|x+\lambda y\|\geq \|x\|$ for every scalar $\lambda$. Also, for any $\varepsilon\in [0,1)$, $x$ is said to be $\varepsilon$-approximate Birkhoff-James orthogonal to $y$, denoted by $x\perp_B^\varepsilon y$, if 
\begin{equation*}
\|x+\lambda y\|^2\geq \|x\|^2-2\varepsilon\|x\|\|\lambda y\|, \qquad \mathrm{for~all~scalars~}\lambda.    
\end{equation*}
For any $\rho >0$, a unitary operator $U$ acting on a Hilbert space $\mathcal K$ is said to be a unitary $\rho$-dilation of an operator $T$ on a Hilbert space $\mathcal H$ if $\mathcal H \subseteq \mathcal K$ and $T^n= \rho P_{\mathcal H} U^n|_{\mathcal H}$ for every nonnegative integer $n$, where $P_{\mathcal H}: \mathcal K \rightarrow \mathcal H$ is the orthogonal projection. Also, when $\rho=1$ and $T$ is a contraction, $U$ is called a unitary dilation of $T$. We obtain the following main results.

\medskip

\begin{enumerate}

\item We find necessary and sufficient conditions such that for any two contractions $T, A$ on $\mathcal H$, their Sch\"{a}ffer unitary dilations $\widetilde{U_T}$ and $\widetilde{U_A}$ on the space $\oplus_{-\infty}^{\infty} \,\mathcal H$ are Birkhoff-James orthogonal. Also, counter example shows that in general $\widetilde{U_T} \not\perp_B \widetilde{U_A}$ even if $T \perp_B A$.

\medskip

\item For any $\rho>0$ and for two Hilbert space operators $T,A$ with $T \perp_B A$, we show that if $\|T\|=\rho$ then $U_T \perp_B U_A$ for any unitary $\rho$-dilations $U_T$ of $T$ and $U_A$ of $A$ acting on a common Hilbert space. Also, we show by an example that the condition that $\|T\|=\rho$ cannot be ignored.

\medskip 

\item For any $\rho >0$, we explicitly construct examples of Hilbert space operators $T, A$ such that $T \not\perp_B A$ but any of their unitary $\rho$-dilations $U_T, U_A$ acting on a common Hilbert space are Birkhoff-James orthogonal.

\medskip

\item We find a characterization for the $\varepsilon$-approximate Birkhoff-James orthogonality of operators on complex Hilbert spaces.

\medskip

\item For any $\rho >0$ and for any Hilbert space operators $T,A$, we find a sharp bound on $\varepsilon$ such that $T \perp_B A$ implies $U_T \perp_B^\varepsilon U_A$ for any unitary $\rho$-dilations $U_T$ of $T$ and $U_A$ of $A$ acting on a common space. Also, we show by an example that in general the bound on $\varepsilon$ cannot be improved.

\medskip

\item We construct families of generalized Sch\"{a}ffer-type unitary dilations for a Hilbert space contraction in two different ways. Then we show that one of them preserves Birkhoff-James orthogonality while any two members $\mathbb U_T, \mathbb U_A$ from the other family are always Birkhoff-James orthogonal irrespective of the orthogonality of $T$ and $A$.

\medskip

\item We show that And\^{o} dilation of a pair of commuting contractions of the form $(T,ST)$, where $S$ is a unitary that commutes with $T$, are orthogonal. Also, we explore orthogonality of regular unitary dilation of a pair of commuting contractions. However, Birkhoff-James orthogonality is independent of commutativity of operators.

\end{enumerate}

	\end{abstract}

	\maketitle

\tableofcontents

\section{Introduction}

\vspace{0.4cm}

\noindent Throughout the paper, all operators are bounded linear operators acting on separable complex Hilbert spaces unless and otherwise stated. For a Hilbert space $\mathcal{H}$, $\mathcal{B}(\mathcal{H})$ denotes the algebra of operators acting on $\mathcal{H}$ endowed with the operator norm. A contraction is an operator with norm not greater than $1$. The set of unit vectors in a normed space $X$ will be denoted by $S_X$. The set of integers and positive integers are denoted by $\mathbb Z$ and $\mathbb N$ respectively. Orthogonality of any two vectors in a normed space will always mean the Birkhoff-James orthogonality of them. Also, orthogonality of two operators are considered only when they act on the same Hilbert space.

\medskip

In this article, we investigate the Birkhoff-James orthogonality of unitary dilations of a pair of orthogonal contractions, or even more generally the orthogonality of unitary $\rho$-dilations of two orthogonal operators. Also, we contribute individually to these two rich classical concepts by finding a characterization of $\varepsilon$-approximate Birkhoff-James orthogonality of a pair of (complex) Hilbert space operators and on the other hand by constructing different explicit Sch\"{a}ffer-type unitary dilations for a contraction. We provide conclusive examples whenever Birkhoff-James orthogonality is not preserved or two dilations are orthogonal without the basic operators being orthogonal. It is well-known that Birkhoff-James orthogonality in $\mathcal B(\mathcal H)$ is independent of the commutativity of two operators, for example the identity operator $I_{\mathcal H}$ commutes with every operator but is not orthogonal to all operators (e.g. $-I_{\mathcal H}$). However, we will also see interplay between Birkhoff-James orthogonality and unitary dilations of a pair of commuting contractions. We begin with a brief introduction on each of these two fundamental concepts in operator theory. 

\subsection{Birkhoff-James orthogonality}Birkhoff-James orthogonality \cite{Birkhoff, James1} in a normed space is a generalization of the inner product orthogonality in an inner product space. Given any two elements $x,y$ in a normed space $\mathbb{X}$, $x$ is said to be \textit{Birkhoff-James orthogonal} to $y$, denoted by $x\perp_B y$, if
\[
\|x+\lambda y\|\geq \|x\|, \quad \mathrm{for~every~scalar}~\lambda.
\]
If $\mathbb X$ is a Hilbert space $x\perp_B y$ if and only if $\langle x, y \rangle = 0$. Also, it easily follows that Birkhoff-James orthogonality is homogeneous, i.e. $x\perp_B y$ if and only if $\alpha x \perp_B \beta y$, for all scalars $\alpha, \beta$. However, unlike the inner product orthogonality, Birkhoff-James orthogonality is neither symmetric nor additive in general  \cite{James1, James2}. To study the symmetric structure of a norm with respect to Birkhoff-James orthogonality, the concept of left-symmetric and right-symmetric points were introduced in \cite{Sain1}. A point $x$ in a normed space $\mathbb{X}$ is said to be left-symmetric (right-symmetric) if $x\perp_B y$ (or, $y\perp_B x$ respectively) implies that $y\perp_B x$ ($x\perp_B y$), for all $y\in \mathbb{X}$.
Birkhoff-James orthogonality in the spaces of operators was first characterized by Magajna \cite{Magajna}. The finite-dimensional version of this result is more popularly known as Bhatia-\v{S}emrl Theorem (see \cite{BS}).

\begin{theorem} [\cite{Magajna}, Lemma 2.2 \& \cite{BS}, Theorem 1.1] \label{Bhatia-Semrl}

Let $T,A \in \mathcal B(\mathcal H)$ for a finite-dimensional Hilbert space $\mathcal{H}$. Then $T\perp_B A$ if and only if there is a unit vector $x_0$ in $\mathcal{H}$ such that $\|Tx_0\|=\|T\|$ and $\langle Tx_0, Ax_0\rangle=0$. Moreover, if $\mathcal{H}$ is infinite-dimensional, then $T\perp_B A$ if and only if there exists a sequence of unit vectors $(x_n)\subseteq \mathcal{H}$ such that 
\[
\lim_{n\rightarrow \infty}\|Tx_n\|=\|T\| \quad \mathrm{and} \quad\lim_{n\rightarrow \infty} \langle Tx_n, Ax_n\rangle=0.
\]
\end{theorem}
A sequence $(x_n)$ is called a \textit{norming sequence} for an operator $T$ if ${\displaystyle \lim_{n\rightarrow \infty}\|Tx_n\|=\|T\| }$. There are several independent proofs of the above theorem in the literature, e.g. \cite{Bhattacharyya & Grover, Keckic, Roy & Bagchi, RSS, Sain & Paul, SPM}. It follows from \cite{Turnsek} that the right-symmetric points in $\mathcal{B}(\mathcal{H})$ are the scalar multiples of isometries or co-isometries, whereas the only left-symmetric point in $\mathcal{B}(\mathcal{H})$ is the zero operator. Birkhoff-James orthogonality turned out to be effctive and useful in studying the geometry of spaces of operators \cite{James3,  Roy & Sain}. Recently, Chmieli\'{n}ski et al. \cite{CW, CSW} extended the notion of Birkhoff-James orthogonality to $\varepsilon$-approximate Birkhoff-James orthogonality.
\begin{definition}
For any $\varepsilon\in [0,1)$ and for any two elements $x,y$ in a normed space $\mathbb{X}$, we say that $x$ is $\varepsilon$-\textit{approximate Birkhoff-James orthogonal} to $y$, denoted by $x\perp_B^\varepsilon y$, if 
\begin{equation}\label{Definition of Approx. orthogonality}
\|x+\lambda y\|^2\geq \|x\|^2-2\varepsilon\|x\|\|\lambda y\|, \quad \mathrm{for~all~scalars~}\lambda.    
\end{equation}
\end{definition}
It is easy to see that the approximate orthogonality is homogeneous and for $\varepsilon=0$ it coincides with Birkhoff-James orthogonality. We learn from the literature (e.g. see Theorem 2.3 in \cite{CSW}) that for any $\varepsilon \in [0,1)$ and for any $x,y$ in a real normed space $\mathbb{X}$, $ x\perp_B^\varepsilon y$ if and only if there exists a functional $f\in\mathbb{X}^*$ with $\|f\|=1$ such that $f(x)=\|x\|$ and $|f(y)|\leq \varepsilon\|y\|$.
A characterization for approximate orthogonality of complex-matrices was obtained in \cite{Roy & Sain II}. Also, an interested reader is referred to \cite{Woijcik} for further reading on approximate orthogonality.

\subsection{Unitary and unitary $\rho$- dilations}
A contraction $T$ acting on a Hilbert space $\mathcal{H}$ is said to possess a \textit{unitary dilation} (\textit{or, an isometric dilation}) if there is a unitary (or, an isometry) $U$ on a Hilbert space $\mathcal{K}\supseteq\mathcal{H}$ such that
\[
T^n=P_{\mathcal{H}}U^n|_\mathcal{H}\quad \text{ for all }\;\; n\in \mathbb N \cup \{ 0 \},
\]
where $P_\mathcal{H}:\mathcal{K} \to \mathcal H$ is the orthogonal projection. Moreover, such a unitary dilation is called \textit{minimal} if
\begin{equation} \label{eqn:sec-1-02}
\mathcal{K}=\bigvee_{n=-\infty}^\infty U^n\mathcal{H}=\overline{\mathrm{span}}\{U^nh:~h\in \mathcal{H},~n\in \mathbb{Z}\}.
\end{equation}
Minimality of an isometric dilation is defined in an analogous manner as in (\ref{eqn:sec-1-02}) by considering $n$ from the set of non-negative integers. The literature (e.g. \cite{NFBK}) tells us that minimal unitary (or, isometric) dilation of a contraction is unique upto unitary equivalence.
One of the foundational results in operator theory due to Sz.-Nagy \cite{Nagy} states that every contraction possesses a minimal unitary dilation. Later, Sch\"{a}ffer \cite{Schaffer} showed an explicit construction (as in Equation-(\ref{eqn:uni-dil})) of unitary dilation of a contraction.

The notion of unitary dilation of a contraction was further extended to unitary $\rho$-dilation by introducing $\zeta_\rho$ class of operators \cite{Berger, Nagy & Foias}. Given any $\rho > 0$, an operator $T$ on a Hilbert space $\mathcal{H}$ is said to belong to class $\zeta_\rho$ or simply to possess a \textit{unitary} $\rho$-\textit{dilation} if there exists a unitary operator $U$ on some Hilbert space $\mathcal{K}\supseteq \mathcal{H}$ such that
\[
T^n=\rho P_{\mathcal{H}}U^n|_\mathcal{H},\quad n\in \mathbb N \cup \{0\}.
\]
For any $T\in \zeta_\rho$, it is easy to see that $\|T\|\leq \rho$. It follows from \cite{Durszt} that the class $\zeta_\rho$ is strictly increasing, i.e., $\zeta_\rho \subsetneq \zeta_{\rho'}$ for $0<\rho<\rho'$. For $\rho=1$, a unitary $\rho$-dilation is just a unitary dilation of a contraction which is to say that the class $\zeta_1$ precisely consists of all contractions. For further details of $\zeta_\rho$ classes of operators, a keen reader is referred to \cite{Nagy & Foias, NFBK}.

\subsection{Motivation and main results of the paper} In Lemma \ref{Orthogonality of extensions}, we show that if two operators $T,A \in \mathcal B(\mathcal H)$ satisfy $T \perp_B A$, then any norm preserving extension $\widetilde{T}$ on $\mathcal K \supseteq \mathcal H$ of $T$ is Birkhoff-James orthogonal to any extension $\widetilde{A} \in \mathcal B(\mathcal K)$ of $A$. This triggers a natural question: what if we replace extensions of $T, A$ by unitary $\rho$-dilations of $T,A$ acting on a common space for an appropriate $\rho >0$ ? This motivates us to explore the connection between Birkhoff-James orthogonality of two operators and that of their respective unitary $\rho$-dilations.

\medskip

In Theorem \ref{Orthogonality of rho dilations}, which is a main result of this paper, we prove that if $T\perp_B A$ and $\|T\|=\rho$ then any unitary $\rho$-dilations $U_T,U_A$ (acting on a common space) of $T$ and $A$ respectively are orthogonal. Also, Example \ref{Example: I} shows that the condition that $\|T\|=\rho$ is crucial and cannot be ignored. However, orthogonality of $T$ and $A$ is not necessary for having orthogonal unitary $\rho$-dilations. In Example \ref{Example for the converse}, we construct a pair of scalar matrices $T,A$ such that $T \not\perp_B A$ but any two unitary $\rho$-dilations $U_T, U_A$ of $T$ and $A$ respectively acting on a common space are Birkhoff-James orthogonal. Thus, it follows from Theorem \ref{Orthogonality of rho dilations} that Sch\"{a}ffer unitary dilations (as in Equation-(\ref{eqn:uni-dil})) of two orthogonal contractions $T,A$ are orthogonal if $\|T\|=1$. So, it is legitimate to ask in general when Sch\"{a}ffer unitary dilations of a pair of contractions are orthogonal. In Theorem \ref{Orthogonality of Halmos Blocks}, we find necessary and sufficient conditions such that the Sch\"{a}ffer unitary dilations of a pair of contractions are orthogonal. We move to seek a converse to this result, i.e. if orthogonality of unitary $\rho$-dilations $U_T$ of $T$ and $U_A$ of $A$ implies the orthogonality of $T$ and $A$. In Example \ref{Example for the converse} we provide a pair of matrices $T,A$ such that $T \not\perp_B A$ but any unitary $\rho$-dilations $U_T,U_A$ (acting on a common space) of $T,A$ respectively are orthogonal. 

In Section \ref{sec:07}, we construct families of generalized Sch\"{a}ffer-type unitary dilations for a contraction in two different ways. Then we show that one of them preserves Birkhoff-James orthogonality while any two members $\mathbb U_T, \mathbb U_A$ from the other family are always Birkhoff-James orthogonal irrespective of the orthogonality of $T$ and $A$.

\medskip

 Even though Birkhoff-James orthogonality does not depend upon the commutativity of two operators as we already mentioned, we investigate in Section \ref{sec:04} orthogonality of And\^{o} dilations of a pair of commuting contractions. We show that And\^{o}'s commuting isometric dilation $(V_T, V_{ST})$ of a commuting pair of contractions of the form $(T, ST)$ satisfies $V_T\perp_B V_{ST}$. Also, in the same Section we show that a pair of orthogonal commuting contractions $(T_1,T_2)$ always possesses an orthogonal regular unitary dilation $(U_1,U_2)$.

In Theorem \ref{Basic: II}, we find a characterization for the $\varepsilon$-approximate Birkhoff-James orthogonality of operators on complex Hilbert spaces. Then, in Theorem \ref{Approximate orthogonality of dilation} we apply this characterization to show that if $T\perp_B A$ then $U_T \perp_B^\varepsilon U_A$ for any unitary $\rho$-dilations $U_T$ of $T$ and $U_A$ of $A$ acting on a common space. Indeed, for any $\rho >0$, we find a bound on $\varepsilon$ such that $T \perp_B A$ implies $U_T \perp_B^\varepsilon U_A$. Also, we show by an example that in general this bound on $\varepsilon$ is sharp and cannot be improved.

Section \ref{sec:02} consists of a few preparatory and basic results.

\vspace{0.2cm}

\section{A few preparatory results on dilation, extension and orthogonality} \label{sec:02}

\vspace{0.3cm}

\noindent In this Section, we prove a few preparatory results on orthogonality of extensions and dilations of operators. Some of these results will be used in sequel. 

\begin{Lemma}\label{Orthogonality of extensions}
Let $T$ and $A$ be two bounded linear operators acting on a Hilbert space $\mathcal{H}$. Let $\widetilde{T}$ be a norm preserving extension of $T$ on some Hilbert space $\mathcal{K}\supseteq \mathcal{H}$. Then $T\perp_B A$ implies that $\widetilde{T}\perp_B \widetilde{A}$, for any extension $\widetilde{A}$ of $A$ on $\mathcal{K}$.
\end{Lemma}

\begin{proof}
Since $T\perp_B A$, there exists a norming sequence $(x_n)$ of $T$ such that
\[\lim_{n \rightarrow \infty } \langle Tx_n, Ax_n\rangle_\mathcal{H} =0.\] Evidently, $(x_n)$ is a norming sequence of $\widetilde{T}$, as
\[\lim_{n\rightarrow \infty }\left\|\widetilde{T}x_n\right\|=\lim_{n\rightarrow \infty }\left\|Tx_n\right\|=\|T\|=\left\|\widetilde{T}\right\|.\]
Also, 
\[\lim_{n\rightarrow \infty } \left \langle \widetilde{T}x_n, \widetilde{A}x_n \right\rangle_\mathcal{K}=\lim_{n\rightarrow \infty } \langle Tx_n, Ax_n\rangle_\mathcal{H} =0.\]
Consequently, $\widetilde{T}\perp_B \widetilde{A}$.
\end{proof}

The following corollary is an easy application of the above lemma.

\begin{cor}
Let $V_1$ and $V_2$ be two isometries on a Hilbert space $\mathcal{H}$ with $V_1\perp_B V_2$. Then $U_1\perp_B U_2$, for every unitary extensions $U_1$ of $V_1$ and $U_2$ of $V_2$ on some Hilbert space $\mathcal{K}\supseteq \mathcal{H}$.
\end{cor}

Given any two members $T$ and $A$ of $\mathcal{B}(\mathcal{H})$, $T\perp_B A$ does not necessarily imply $T^k\perp_B A$, for all natural numbers $k$. However, the implication is true when $T$ is a self-adjoint operator. We establish this fact using the following lemma and few corollaries of it. This in particular facilitates a nice orthogonality relation between $D_T$ and $T$, whenever $\|D_T\|=1$.

\begin{Lemma}\label{Lemma about self-adjoint}
Let $\mathcal{H}$ be a Hilbert space and let $T$ be a self-adjoint operator on $\mathcal{H}$. Then
\begin{itemize}
    \item[(i)] for any norming sequence $(x_n)$ of $T^2$, ${\displaystyle \lim_{n\rightarrow \infty} \|T^{2k}x_n-\|T\|^{2k}x_n\|=0}$, for all integers $k\geq 0$;
    
    \item[(ii)] for any norming sequence $(x_n)$ of $T$, ${\displaystyle \lim_{n\to \infty}\|T^{2k+1}x_n-\|T\|^{2k}Tx_n\|=0}$, for all integers $k\geq 0.$
\end{itemize}
\end{Lemma}

\begin{proof}
(i) Observe that $(x_n)$ is a norming sequence of $T$, since $\|T\|^2=\|T^2\|$ and
\[
\|T^2x_n\|\leq \|T\|\|Tx_n\| \quad \text{ for all } \;\; n\in \mathbb{N}.
\]
This is because, if every subsequence of a bounded sequence converges to a limit then so does the original sequence. We now apply induction to prove our assertion. The claim is trivially true for $k=0$. The claim is also true for $k=1$, since
\[
\|\|T\|^2x_n-T^2x_n\|^2  = \|T\|^4+\|T^2x_n\|^2-2\|T\|^2\|Tx_n\|^2\to 0 \quad \mathrm{as} \; \; ~n\to \infty.
\]
Suppose that the claim is true for $k=m$. Now for $k=m+1$, we have
\begin{align*}
\|T^{2m+2}x_n-\|T\|^{2m+2}x_n\| & = \|T^{2m+2}x_n-\|T\|^2T^{2m}x_n+\|T\|^2T^{2m}x_n-\|T\|^{2m+2}x_n\|\\
& \leq \|T^{2m}(T^2-\|T\|^2)x_n+\|T\|^2(T^{2m}-\|T\|^{2m})x_n\|\\
& \leq \|T^{2m}\|\|T^2x_n-\|T\|^2x_n\|+\|T\|^2\|T^{2m}x_n-\|T\|^{2m}x_n\|.
\end{align*}
By inductive hypothesis we have 
\[
\lim_{n\rightarrow \infty }\|T^{2m}x_n-\|T\|^{2m}x_n\|=0.
\]
Consequently,
\[
\lim_{n\rightarrow \infty } \|T^{2m+2}x_n-\|T\|^{2m+2}x_n\|=0
\]
and the assertion is proved.

\bigskip

(ii)  For any norming sequence $(x_n)$ of $T$,
\begin{align*}
\|T\|^2-\|T^2x_n\| & \leq \|T\|^2-\langle T^2x_n,x_n \rangle \qquad [ \text{since } \,T^2 ~\mathrm{is~positive}]\\
& = \|T\|^2-\|Tx_n\|^2 \; \to 0 \quad \mathrm{as} \;\; ~n\to \infty.
\end{align*}
Thus, $(x_n)$ is a norming sequence of $T^2$. Note that for any integer $k\geq 0$
\[
\|T^{2k+1}x_n-\|T\|^{2k}Tx_n\|\leq \|T\|\|T^{2k}x_n-\|T\|^{2k}x_n\|.
\]
Since $(x_n)$ is also a norming sequence of $T^2$, by virtue of part (i)
\[\|T^{2k+1}x_n-\|T\|^{2k}Tx_n\|\leq \|T\|\|T^{2k}x_n-\|T\|^{2k}x_n\| \to 0 \quad \mathrm{as}\;\; ~n\to \infty.\]
This completes the proof.
\end{proof}

An important observation in this connection is the following corollary.

\begin{cor}\label{norming sequence}
Let $\mathcal{H}$ be a Hilbert space and let $T$ be a self-adjoint operator on $\mathcal{H}$. Then $\|T^k\|=\|T\|^k$ for all natural numbers $k$. Moreover, any norming sequence of $T$ is also a norming sequence of $T^k$ for all $k\in \mathbb N$.
\end{cor}

\begin{proof}
Consider any norming sequence $(x_n)$ of $T$. By virtue of Lemma \ref{Lemma about self-adjoint}
\[
\|T^{2k+1}x_n\|-\|T\|^{2k}\|Tx_n\|\leq \|T^{2k+1}x_n-\|T\|^{2k}Tx_n\|\to 0 \quad \mathrm{as} \; \; ~n\to \infty
\]
for all integers $k\geq 0$. Thus,
\[
\|T^{2k+1}\|\geq\lim_{n\to \infty}\|T^{2k+1}x_n\|=\lim_{n\to \infty}\|T\|^{2k}\|Tx_n\|=\|T\|^{2k+1}\geq \|T^{2k+1}\|.
\]
Therefore, the above inequality is an equality, and we have $\|T^{2k+1}\|=\|T\|^{2k+1}$. Again, it follows from Lemma \ref{Lemma about self-adjoint} that
$(x_n)$ is also a norming sequence of $T^2$. Also,
\[
\|T^{2k}x_n\|-\|T\|^{2k}\|x_n\|\leq \|T^{2k}x_n-\|T\|^{2k}x_n\|\to 0 \quad \mathrm{as} \; \; ~n\to \infty
\]
for all integers $k\geq 0$. Thus, ${\displaystyle \lim_{n\to \infty}\|T^{2k}x_n\|=\|T\|^{2k}}$ and by an analogous argument as above we have $ \|T^{2k}\|=\|T\|^{2k}$. Consequently, $\|T^k\|=\|T\|^k$ for all $k\in \mathbb{N}$.
Moreover, since the norming sequence $(x_n)$ of $T$ was chosen arbitrarily, it follows that every norming sequence of $T$ is also a norming sequence of $T^k$ for all integers $k\geq 0$.

\end{proof}

The following corollary is another application of Lemma \ref{Lemma about self-adjoint}.

\begin{cor}\label{Corollary about self-adjoint}
Let $\mathcal{H}$ be a Hilbert space and let $T$ be a self-adjoint operator on $\mathcal{H}$. Then the following hold:
\begin{itemize}
    \item[(i)] $T^2\perp_B A$ for any $A\in \mathcal{B}(\mathcal{H})$ implies that $T^{2k}\perp_B A$, for all integers $k\geq 0.$
    \item[(ii)] $T\perp_B A$ for any $A\in \mathcal{B}(\mathcal{H})$ implies that $T^{2k+1}\perp_B A$, for all integers $k\geq 0.$
\end{itemize}
\end{cor}

\begin{proof}
(i) Since $T^2\perp_B A$, we can find a norming sequence $(x_n)$ of $T^2$ such that 
\[
\lim_{n\rightarrow \infty } \langle T^2x_n, Ax_n \rangle = 0.
\]
By Lemma \ref{Lemma about self-adjoint}, we have
\begin{align*}
|\langle T^{2}x_n, Ax_n \rangle - \langle \|T\|^{2}x_n, Ax_n \rangle | & = |\langle T^{2}x_n-\|T\|^{2}x_n, Ax_n \rangle|\\
& \leq \|T^{2}x_n-\|T\|^{2}x_n\|\|A\| \to 0 \quad \mathrm{as}~n\to \infty.   
\end{align*}
This shows that ${ \displaystyle \lim_{n\rightarrow \infty }  \langle \|T\|^{2}x_n, Ax_n \rangle=0 }$ and hence ${ \displaystyle \lim_{n\rightarrow \infty } \langle x_n, Ax_n \rangle =0}$.
Again, by Lemma \ref{Lemma about self-adjoint} we have
\begin{align*}
|\langle T^{2k}x_n, Ax_n \rangle - \langle \|T\|^{2k}x_n, Ax_n \rangle |& = |\langle T^{2k}x_n-\|T\|^{2k}x_n, Ax_n \rangle|\\
& \leq \|T^{2k}x_n-\|T\|^{2k}x_n\|\|A\|\to 0   \quad \mathrm{as} \; \; ~n\to \infty.   
\end{align*}
Since ${0= \displaystyle \lim_{n\to \infty}\langle x_n, Ax_n \rangle=\lim_{n\to \infty} \langle \|T\|^{2k}x_n, Ax_n \rangle}$,  we have that
$
{\displaystyle \lim_{n\rightarrow \infty} \langle T^{2k}x_n, Ax_n \rangle =0}$ for all $k \in \mathbb N$.
Moreover, it follows from Corollary \ref{norming sequence} that $(x_n)$ is also a norming sequence of $T^{2k}$ for each integer $k\geq 0$. Therefore, $T^{2k}\perp_B A$, for all integers $k\geq 0$.

\medskip

(ii) Since $T\perp_B A$, there exists a norming sequence $(x_n)$ of $T$ such that 
\[
\lim_{n\to \infty}\langle Tx_n, Ax_n \rangle = 0.
\]
By Lemma \ref{Lemma about self-adjoint}
\begin{align*}
|\langle T^{2k+1}x_n, Ax_n \rangle - \langle \|T\|^{2k}Tx_n, Ax_n \rangle | & = |\langle T^{2k+1}x_n-\|T\|^{2k}Tx_n, Ax_n \rangle|\\
& \leq \|T^{2k+1}x_n-\|T\|^{2k}Tx_n\|\|A\|\to 0   \quad \mathrm{as} \; \; ~n\to \infty.   
\end{align*}
Since ${\displaystyle 0=\lim_{n\to \infty} \langle Tx_n, Ax_n \rangle =\lim_{n\to \infty } \langle \|T\|^{2k}Tx_n, Ax_n \rangle }$, we have
$
{\displaystyle \lim_{n\to \infty} \langle T^{2k+1}x_n, Ax_n \rangle =0.}
$
Again, by Corollary \ref{norming sequence}, $(x_n)$ is a norming sequence of $T^{2k+1}$ for each integer $k\geq 0$. Thus, $T^{2k+1}\perp_B A$ for each integer $k\geq 0$.
\end{proof}

Finally, using the forgoing corollaries we have the following orthogonality relation between $D_T$ and $T$.

\begin{cor}\label{Orthogonality of T and DT}
Let $\mathcal{H}$ be a Hilbert space and $T$ be a contraction on $\mathcal{H}$ such that $\|D_T\|=1$. Then $D_T^k\perp_B T^j$ for all integers $k,j>0$.
\end{cor}

\begin{proof}
Let $(x_n)$ be a norming sequence of $D_T$. Then 
\[
1=\lim_{n\to \infty} \|D_Tx_n\|^2=\lim_{n\to \infty}\langle D_T^2x_n,x_n\rangle=\lim_{n\to \infty }(1-\|Tx_n\|^2).
\]
Consequently, ${ \displaystyle \lim_{n\to \infty}\|Tx_n\|^2\ = 0 }$. Now, for positive integers $k$ and $j$
\[
|\langle D_T^kx_n,T^jx_n\rangle|\leq \|D_T^kx_n\|\|T^jx_n\|\leq \|T^jx_n\|\leq \|Tx_n\|\to 0 \quad \mathrm{as} \; \; ~n\to \infty.
\]
Since $D_T$ is a self-adjoint operator, $(x_n)_{n\in \mathbb{N}}$ is a norming sequence of $D_T^k$, for all positive integers $k$ (see Corollary \ref{norming sequence}). Thus, $D_T^k\perp_B T^j$ for all integers $k,j>0$. This completes the proof.
\end{proof}


\section{Birkhoff-James orthogonality of Sch\"{a}ffer's unitary dilations} \label{sec:03}

\vspace{0.4cm}

\noindent Recall that a unitary (or isometry) $U$ acting on a Hilbert space $\mathcal K$ is said to be a unitary (or isometric) dilation of a contraction $T$ on a Hilbert space $\mathcal{H}$ if $\mathcal H \subseteq \mathcal K$ and 
\[
T^n=P_{\mathcal{H}}U^n|_\mathcal{H},\qquad n\in \mathbb N \cup \{0\},
\]
where $P_\mathcal{H}$ is the orthogonal projection of $\mathcal{K}$ onto $\mathcal{H}$. In \cite{Schaffer}, Sch\"{a}ffer found the following explicit unitary dilation $\widetilde{U_T}$ on the space $\mathcal K = \oplus_{-\infty}^{\infty} \mathcal H_i$, where $\mathcal H_i =\mathcal H$ for every $i$, for a contraction $T$ on a Hilbert space $\mathcal H$. This is widely known as \textit{Sch\"{a}ffer unitary dilation} of a contraction.

\begin{equation} \label{eqn:uni-dil}
	\widetilde{U_T} =\left[
\begin{array}{ c c c c|c|c c c c}
\bm{\ddots}&\vdots &\vdots&\vdots   &\vdots  &\vdots& \vdots&\vdots&\vdots\\
\cdots&0&\mathrm{I}&0  &0&  0&0&0&\cdots\\
\cdots&0&0&\mathrm{I}  &0&  0&0&0&\cdots\\
\cdots&0&0&0  &\mathrm{D_T}&  \mathrm{-T^*}&0&0&\cdots\\ \hline

\cdots&0&0&0   &\mathrm{T}&   \mathrm{D_{T^*}}&0&0&\cdots\\ \hline

\cdots&0&0&0   &0&  0& \mathrm{I}&0&\cdots\\
\cdots&0&0&0   &0&  0&0&\mathrm{I}&\cdots\\
\cdots&0&0&0  &0&   0& 0&0&\cdots\\
\vdots&\vdots&\vdots&\vdots&\vdots&\vdots&\vdots&\vdots&\bm{\ddots}\\
\end{array} \right].
	\end{equation}
We will denote by $\widetilde{U_T}(i,j)$ the $(i,j)^{th}$ entry of the block matrix of $\widetilde{U_T}$. The operators $D_T$ and $D_{T^*}$ are the defect operators of $T$ and $T^*$ respectively, defined as $(I-T^*T)^\frac{1}{2}$ and $(I-TT^*)^\frac{1}{2}$. The $2 \times 2$ block
\begin{equation}\label{Halmos block}
\mathsf{T_U}=\begin{bmatrix}
D_T & -T^*\\
T & D_{T^*}
\end{bmatrix}    
\end{equation}
in the block matrix of $\widetilde{U_T}$ is also a unitary operator and is known as \textit{Halmos block} associated with $\widetilde{U_T}$.\\

In this Section, we find necessary and sufficient conditions such that the Sch\"{a}ffer unitary dilations of two contractions are orthogonal. Indeed, the Sch\"{a}ffer unitary dilations of a pair of orthogonal contractions may not always be orthogonal. This is a main result of this Section. In this connection we would like to mention that for any bounded linear operator $T$ on a Hilbert space $\mathcal{H}$, the \textit{maximal numerical range} of $T$ is defined by
\begin{equation} \label{eqn:sec3-03}
\mathcal{W}(T):=\{\lambda\in \mathbb{C}:~\lim_{n\to \infty}\langle Tx_n, x_n\rangle\to \lambda,~\|x_n\|=1~ \& \; \lim_{n\to \infty}\|Tx_n\|= \|T\|\}.
\end{equation}
It is well-known \cite{Stampfli} that the maximal numerical range of an operator is non-empty, closed, convex and contained in the closure of its numerical range.

\begin{theorem}\label{Orthogonality of Halmos Blocks}
Let $T, A$ be any two contractions acting on a Hilbert space $\mathcal{H}$. Let $\widetilde{U_T}$ and $\widetilde{U_A}$ be the Sch\"{a}ffer unitary dilations as in Equation-$($\ref{eqn:uni-dil}$)$ of $T$ and $A$ respectively on $\mathcal{K}=\bigoplus_{-\infty}^\infty \mathcal{H}$. Also, let $\mathsf{T}_U$, $\mathsf{A}_U$ be the $2\times 2$ Halmos-blocks associated with $\widetilde{U_T}$ and $\widetilde{U_A}$ respectively as in Equation-$(\ref{Halmos block})$. Then the following are equivalent:
\begin{itemize}
\item[(i)] $\widetilde{U_T}\perp_B \widetilde{U_A}$ ;

\item[(ii)] $\mathsf{T}_U \perp_B \mathsf{A}_U$ ;

\item[(iii)] $\mathcal{W}(\mathsf{A}_U^* \mathsf{T}_U)$ contains a negative real number.
\end{itemize}
\end{theorem}

\begin{proof}

$(i) \Rightarrow (ii)\;\; \& \;\; (i) \Rightarrow (iii)$. Suppose that $\widetilde{U_T} \perp_B \widetilde{U_A}$. Then there exists a sequence of unit vectors $(\mathbf{h}_k)_{k\in \mathbb{N}}$ in $\mathcal{K}$ such that ${\displaystyle \lim_{k \to \infty }\left\langle \widetilde{U_T} \mathbf{h}_k, \widetilde{U_A} \mathbf{h}_k\right\rangle_\mathcal{K} = 0}$. Let 
\[\mathbf{h}_k=\left(x^{(k)}_n\right)_{n\in \mathbb{Z}}, \qquad k\in \mathbb{N}.\]
Therefore,
\[
\lim_{k\to \infty} \left\langle \widetilde{U_T} \mathbf{h}_k, \widetilde{U_A} \mathbf{h}_k\right\rangle_\mathcal{K}=\lim_{k \to \infty } (M_k+N_k) = 0,
\]
where
\[
M_k = \sum\limits_{n\neq 0,1}\left\|x_n^{(k)}\right\|^2 \quad \mathrm{and} \quad N_k = \left\langle \mathsf{T}_U\left(x_0^{(k)},x_1^{(k)}\right), \mathsf{A}_U \left(x_0^{(k)},x_1^{(k)}\right)\right\rangle_{\mathcal{H}\oplus \mathcal{H}}, \qquad k\in \mathbb{N}.\]
Without loss of generality we can assume that $(M_k)_{k\in \mathbb{N}}$ and $(N_k)_{k\in \mathbb{N}}$ are convergent, otherwise we will choose suitable subsequences of them. Let
\[
\lim_{k\to \infty} M_k = \alpha \quad \mathrm{and} \quad \lim_{k\to \infty} N_k = \beta.
\]
If $\alpha = 1$, then
\[
\lim_{k\to \infty} \left(\left\|x_0^{(k)}\right\|^2+\left\|x_1^{(k)}\right\|^2\right) = 1 - \alpha = 0.
\]
Consequently we have ${ \displaystyle \lim_{k\to \infty } N_k = 0}$ as
\[
0\leq N_k \leq \left(\left\|x_0^{(k)}\right\|^2+\left\|x_1^{(k)}\right\|^2\right), \qquad k\in \mathbb{N}.
\]
This is contradiction to the fact that ${\displaystyle \lim_{k\to \infty}(M_k+N_k) = 0}$. Therefore, $0\leq \alpha < 1$ and
\[
\lim_{k\to \infty} \left(\left\|x_0^{(k)}\right\|^2+\left\|x_1^{(k)}\right\|^2\right) = 1-\alpha > 0.
\]
Without loss of generality let $\left(\left\|x_0^{(k)}\right\|^2+\left\|x_1^{(k)}\right\|^2\right)\neq 0$ for all $k\in \mathbb{N}$. Consider
\[
y_0^{(k)}=\frac{x_0^{(k)}}{\sqrt{\left\|x_0^{(k)}\right\|^2+\left\|x_1^{(k)}\right\|^2}}, \qquad y_1^{(k)}=\frac{x_1^{(k)}}{\sqrt{\left\|x_0^{(k)}\right\|^2+\left\|x_1^{(k)}\right\|^2}}, \qquad k\in \mathbb{N}.
\]
Note that
\[
N_k = \left(\left\|x_0^{(k)}\right\|^2+\left\|x_1^{(k)}\right\|^2\right)\left\langle \mathsf{T}_U\left(y_0^{(k)},y_1^{(k)}\right), \mathsf{A}_U \left(y_0^{(k)},y_1^{(k)}\right)\right\rangle_{\mathcal{H}\oplus \mathcal{H}}, \qquad k\in \mathbb{N}.
\]
Considering limit on the both sides we have
\[
\lim_{k\to \infty} \left\langle \mathsf{T}_U\left(y_0^{(k)},y_1^{(k)}\right), \mathsf{A}_U \left(y_0^{(k)},y_1^{(k)}\right)\right\rangle_{\mathcal{H}\oplus \mathcal{H}} = \frac{-\alpha}{1-\alpha}.
\] 
Thus, whenever $\alpha = 0$ we reach $(ii)$ and for $\alpha \neq 0$ we have $(iii)$.\\

\noindent $(ii)\Rightarrow (i)$. Suppose (ii) holds. Then there exists a sequence of unit vectors $((y_k,y'_k))_{k\in \mathbb{N}}\subseteq \mathcal{H}\bigoplus \mathcal{H}$ such that
\[
\lim_{k\to \infty}\left\langle \mathsf{T}_U (y_k,y'_k), \mathsf{A}_U (y_k,y'_k) \right\rangle_{\mathcal{H}\oplus \mathcal{H}}=0.
\]
Consider $(\mathbf{h}_k)_{k\in \mathbb{N}}$ in $\mathcal{K}$ such that $\mathbf{h}_k=\left(x_n^{(k)}\right)_{n\in \mathbb{Z}}$, where
\[
{x}_n^{(k)}=\begin{cases}
\mathbf{0}, & \mathrm{if}~n\neq 0,1,\\
y_k, & \mathrm{if}~n=0,\\
y'_k & \mathrm{if}~n=1.
\end{cases}
\]
Consequently,
\begin{align*}
\lim_{k\to \infty }\left\langle \widetilde{U_T} \mathbf{h}_k, \widetilde{U_A} \mathbf{h}_k \right\rangle_{\mathcal{K}}=\lim_{k\to \infty}\left\langle \mathsf{T}_U (y_k,y'_k), \mathsf{A}_U (y_k,y'_k) \right\rangle_{\mathcal{H}\oplus \mathcal{H}}=0,
\end{align*}
and $\widetilde{U_T}\perp_B \widetilde{U_A}$.\\

\noindent $(iii)\Rightarrow (i)$. Suppose $(iii)$ holds. Then there exists a sequence of unit vectors $(u_k,v_k)_{k\in\mathbb{N}}$ in $\mathcal{H}\bigoplus \mathcal{H}$ such that
\[
\lim_{k\to \infty}\left\langle \mathsf{A}_U^* \mathsf{T}_U(u_k,v_k),(u_k,v_k)\right\rangle_{\mathcal{H}\oplus \mathcal{H}}= \gamma,
\]
for some $\gamma < 0 $. Choose sequence of unit vectors $(\mathbf{h}_k)_{k\in \mathbb{N}}$ and $(\mathbf{h'}_k)_{k\in \mathbb{N}}$ in $\mathcal{K}$ such that 
\[\mathbf{h}_k=\left(x_n^{(k)}\right)_{n\in \mathbb{Z}} \quad \mathrm{and} \quad \mathbf{h'}_k=\left(y_n^{(k)}\right)_{n\in \mathbb{Z}},\]
where
\[x^{(k)}_0=x^{(k)}_1=\mathbf{0};\qquad y_n^{(k)}=\begin{cases} \mathbf{0}, & \mathrm{if}~n\neq 0,1;\\
u_k, & \mathrm{if}~n= 0;\\
v_k, & \mathrm{if}~n= 1.
\end{cases}\]
Let 
\[\widetilde{\mathbf{h}}_k = \sqrt{\dfrac{-\gamma}{1-\gamma}}\mathbf{h}_k+\dfrac{1}{\sqrt{1-\gamma}}\mathbf{h'}_k, \qquad k\in \mathbb{N}.\]
Then $\left(\widetilde{\mathbf{h}}_k\right)_{k\in \mathbb{N}}$ is a sequence of unit vectors in $\mathcal{K}$, and we have
\begin{align*}
\left\langle \widetilde{U_T} \widetilde{\mathbf{h}}_k, \widetilde{U_A} \widetilde{\mathbf{h}}_k \right\rangle_{\mathcal{K}} & = \frac{-\gamma}{1-\gamma}+\frac{1}{1-\gamma}\left\langle \mathsf{A}_T^* \mathsf{T}_U(u_k,v_k),(u_k,v_k)\right\rangle_{\mathcal{H}\oplus \mathcal{H}}.
\end{align*}
Taking limit on the both sides we have
\[
\lim_{k\to \infty} \left\langle \widetilde{U_T} \widetilde{\mathbf{h}}_k, \widetilde{U_A} \widetilde{\mathbf{h}}_k \right\rangle_{\mathcal{K}}=0,
\]
and consequently $\widetilde{U_T}\perp_B \widetilde{U_A}$. The proof is now complete.

\end{proof}

The following is an obvious corollary of Theorem \ref{Orthogonality of Halmos Blocks}. Indeed, combining conditions $(ii)$ and $(iii)$ of Theorem \ref{Orthogonality of Halmos Blocks}, one can also state Theorem \ref{Orthogonality of Halmos Blocks} in the following way.

\begin{cor}\label{Corollary to Halmos Blocks}
Let $T, A$ be two contractions acting on a Hilbert space $\mathcal{H}$. Let $\widetilde{U_T}$ and $\widetilde{U_A}$ denote the Sch\"{a}ffer dilations of $T$ and $A$ on $\mathcal{K}=\bigoplus_{-\infty}^\infty \mathcal{H}$, respectively. Then $\widetilde{U_T}\perp_B \widetilde{U_A}$ if and only if $\mathcal{W}(\mathsf{A}_U^* \mathsf{T}_U)$ contains a non-positive real number.
\end{cor}

\begin{remark}
Orthogonality of two contractions does not ensure the orthogonality of their Sch\"{a}ffer unitary dilations and we will show this in Example \ref{Example: I}. However, for two contractions $T, A$ if $T \perp_B A$ and if $\|T\|=1$, then their Sch\"{a}ffer unitary dilations are orthogonal. We will prove this result in Section \ref{sec:05} (see Theorem \ref{Orthogonality of rho dilations}) in a more general frame for unitary $\rho$-dilations of $T$ and $A$. We will see there that the condition that $\|T\|=1$ is very crucial in this context.

\end{remark}

Also, the orthogonality of dilations of two contractions does not depend upon the orthogonality of the contractions. Indeed, the following example shows that the Sch\"{a}ffer dilations of two non-orthogonal contractions can be orthogonal.

\begin{example}
Let $\mu$ and $\lambda$ be any two non-zero real numbers with $\lambda^2+\mu^2=1.$ Now, let
\[T=\lambda I_2, \qquad A=\mu I_2,\]
where $I_2$ denotes the identity operator on the two-dimensional complex Hilbert space $\mathcal{H}$. Observe that $T\not \perp_B A$. Also, 
\[T=T^*,\quad D_T=D_{T^*}=\mu I_2, \qquad \mathrm{and} \qquad A=A^*,\quad D_A=D_{A^*}=\lambda I_2.\]
Now, for any $(x_0,y_0)\in S_{\mathcal{H}\oplus \mathcal{H}},$ we have
\begin{align*}
\left\langle \mathsf{U}_T(x_0,y_0), \mathsf{U}_A(x_0,y_0)\right\rangle & =  \left\langle (D_Tx_0-T^*y_0, Tx_0+D_{T^*}y_0), (D_Ax_0-A^*y_0, Ax_0+D_{A^*}y_0)\right\rangle_{\mathcal{H}\oplus \mathcal{H}}\\
& = \left\langle (\mu x_0-\lambda y_0, \lambda x_0+\mu y_0), (\lambda x_0-\mu y_0, \mu x_0+\lambda y_0)\right\rangle_{\mathcal{H}\oplus \mathcal{H}}\\
& =2 \mu \lambda (\|x_0\|^2+\|y_0\|^2)+(\lambda^2-\mu^2)\langle x_0, y_0 \rangle_\mathcal{H} - (\lambda^2-\mu^2) \langle y_0,x_0\rangle_\mathcal{H}\\
& =2\mu \lambda + 2i~\mathfrak{Im}(\lambda^2-\mu^2)  \langle x_0, y_0 \rangle_\mathcal{H}.
\end{align*}
If $\lambda$, $\mu$ are of the opposite signs, then choosing $y_0=\mathbf{0}$ we have
\[\left\langle \mathsf{U}_T(x_0,y_0), \mathsf{U}_A(x_0,y_0)\right\rangle_{\mathcal{H}\oplus \mathcal{H}}= 2\mu \lambda < 0.\]
Consequently, by Theorem \ref{Orthogonality of Halmos Blocks} we have $\widetilde{U_T}\perp_B \widetilde{U_A}$.  \qed

\end{example}

However, if $\lambda$, $\mu$ are of the same signs then
$\mathfrak{Re}\left\langle \mathsf{U}_T(x,y), \mathsf{U}_A(x,y)\right\rangle_{\mathcal{H}\oplus \mathcal{H}} = 2\mu \lambda > 0$ for $(x,y)\in S_{\mathcal{H}\oplus \mathcal{H}}$. In that case $\widetilde{U_T}\not\perp_B \widetilde{U_A}$. In Section \ref{sec:07}, we will construct Sch\"{a}ffer-type unitary dilations $\mathbb U_T, \mathbb U_A$ for any two contractions $T, A$ respectively such that $\mathbb U_T \perp_B \mathbb U_A$ even if $T$ and $A$ are not orthogonal. Before that we show an easier way of finding orthogonal unitary dilations of $T$ and $A$.

\subsection{Orthogonal Sch\"{a}ffer dilations}

Consider the Sch\"{a}ffer dilations $\widetilde{U_T}$ of $T$ and $\widetilde{U_{A^*}}$ of $A^*$ on $\mathcal{K}$ as in Equation-(\ref{eqn:uni-dil}). Then $\widetilde{U_{A^*}}^*$ is a unitary dilation of $A$ on $\mathcal{K}$. We show that $\widetilde{U_T}\perp_B \widetilde{U_{A^*}}^*.$

\medskip

For any unit vector $x\in \mathcal{H}$, consider the vector $\mathbf{h}=(h_n)_{m\in\mathbb{Z}}$ in $\mathcal{K}$ defined by
\[h_n=\begin{cases}
 D_Tx, & \mathrm{if}~n=0;\\
-Tx, & \mathrm{if}~n=1;\\
\mathbf{0}, & \mathrm{otherwise}.
\end{cases}\]
Evidently, $\mathbf{h}$ is a unit vector in $\mathcal{K}$. Let $\widetilde{U_T}\mathbf{h}=(l_n)_{n\in \mathbb{Z }}$. Using $D_T^2+T^*T=I_\mathcal{H}$ and $TD_T=D_{T^*}T$ (see CH-I of \cite{NFBK}), a straightforward computation reveals that 
\[l_n=\begin{cases}
\mathbf{0}, & \mathrm{if}~n\neq -1;\\
 x, & \mathrm{if}~n=-1,
\end{cases}\]

\medskip

On the other hand, $\widetilde{U_{A^*}}^*\mathbf{h}=(p_n)_{n\in \mathbb{Z}}$, where
\[p_n=\begin{cases}
AD_Tx, & \mathrm{if}~n=0;\\
D_AD_Tx, & \mathrm{if}~n=1;\\
-Tx, & \mathrm{if}~n=2;\\
\mathbf{0}, & \mathrm{otherwise}.
\end{cases}\]

\medskip

It is now easy to check that 
\[\left\langle \widetilde{U_T}\mathbf{h},\widetilde{U_{A^*}}^*\mathbf{h}\right\rangle=0.\]
Therefore, $\widetilde{U_T}\perp_B \widetilde{U_{A^*}}^*$. Thus, we have constructed unitary dilations $U_T=\widetilde{U_T}$ of $T$ and $U_A=\widetilde{U_{A^*}}^*$ of $A$ such that $U_T\perp_B U_A$.\qed\\

We conclude this Section with the following easy corollary of the above construction of orthogonal dilations.
\begin{cor}
Let $T$ be a contraction acting on a Hilbert space $\mathcal{H}$ and let $\widetilde{U_T}$ on $\mathcal{K}=\bigoplus_{-\infty}^\infty\mathcal{H}$ be the Sch\"{a}ffer dilation of $T$ as in Equation-(\ref{eqn:uni-dil}). Then $\widetilde{U_T}\perp_B I_\mathcal{K}$.
\end{cor}

\begin{proof}
Consider the same unit vector $\mathbf{h}$ in $\mathcal{K}$ as it was chosen in above construction. It is easy to verify that $\left\langle \widetilde{U_T}\mathbf{h},\mathbf{h}\right\rangle_\mathcal{K}=0.$
Therefore, $\widetilde{U_T}\perp_B I_\mathcal{K}$.
\end{proof}

\section{Orthogonality of unitary $\rho$-dilations} \label{sec:05}

\vspace{0.4cm}

\noindent In this Section, we study the Birkhoff-James orthogonality in a more general setting. More precisely, we investigate if Birkhoff-James orthogonality of elements in $\zeta_\rho$ class is preserved by their respective unitary $\rho$-dilations. Recall that for any $\rho > 0$, an operator $T$ on a Hilbert space $\mathcal{H}$ is said to admit a unitary $\rho$-dilation or in other words $T$ belongs to the class $\zeta_\rho$ if there is a unitary $U$ on some Hilbert space $\mathcal{K}\supseteq \mathcal{H}$ such that
\[
T^n=\rho P_{\mathcal{H}}U^n|_\mathcal{H},\quad \text{ for all }\; n\in \mathbb N \cup \{0\}.
\]
For any $T\in \zeta_\rho$, it is easy to see that $\|T\|\leq \rho$. It follows from \cite{Durszt} that the class $\zeta_\rho$ is strictly increasing, i.e., $\zeta_\rho \subsetneq \zeta_{\rho'}$ for $0<\rho<\rho'$. Needless to mention that the class $\zeta_1$ consists of all contractions. An interested reader is referred to \cite{Nagy & Foias, NFBK} for a further reading on $\zeta_\rho$ class of operators. The first main result of this Section shows that Birkhoff-James orthogonality of any two operators $T,A$ in $\zeta_\rho$ is preserved by any of their respective unitary $\rho$-dilations if $\|T\|=\rho$. Indeed, it turns out that the norm of an element in $\zeta_\rho$ plays a significant role in this connection.

\begin{theorem}\label{Orthogonality of rho dilations}

Let $\mathcal{H}$ be a Hilbert space and $T,A\in \mathcal{B}(\mathcal{H})$. Let $T, A\in \zeta_\rho$ for any $\rho\in (0,\infty)$ with $\|T\|=\rho$. If $T$ and $A$ possess unitary $\rho$-dilations $U_T$ and $U_A$ respectively on a common space $\mathcal{K}\supseteq \mathcal{H}$, then $U_T\perp_B U_A$ whenever $T\perp_B A$.

\end{theorem}

\begin{proof}
Let $U_T$ and $U_A$ be any unitary $\rho$-dilations of $T$ and $A$ respectively on a space $\mathcal{K}\supseteq \mathcal{H}$. For any sequence of unit vectors $(x_n)$ in $\mathcal{H}$, we have
\begin{align*}
\|Tx_n-\rho U_Tx_n\|^2 & = \langle Tx_n,Tx_n\rangle_\mathcal{K}-2~\mathfrak{Re} \langle Tx_n, \rho U_Tx_n \rangle_\mathcal{K}+\rho^2\langle U_Tx_n,U_Tx_n\rangle_\mathcal{K}\\
& = \|Tx_n\|^2-2~\mathfrak{Re} \langle Tx_n, \rho P_\mathcal{H}U_Tx_n \rangle_\mathcal{K}+\rho^2\\
& = \|Tx_n\|^2-2\|Tx_n\|^2+\rho^2\\
& = \rho^2-\|Tx_n\|^2.
\end{align*}
In particular, for a norming sequence $(x_n)$ of $T$ we have 
\begin{equation}\label{Eq:1}
\lim_{n\to \infty } \|Tx_n-\rho U_Tx_n\|^2 = \lim_{n\to \infty } (\rho^2-\|Tx_n\|^2) = 0.
\end{equation}
Since $T\perp_B A$, there is a norming sequence $(y_n)$ of $T$ such that
\[\lim_{n\to \infty }\langle Ty_n,Ay_n \rangle_\mathcal{H} = 0.\]
Note that 
\begin{align*}
\left|\langle \rho U_T y_n, \rho U_A y_n\rangle_\mathcal{K} - \langle T y_n, A y_n\rangle_\mathcal{H} \right| & = \left|\langle \rho U_T y_n, \rho U_A y_n\rangle_\mathcal{K} - \langle T y_n, \rho P_\mathcal{H}U_A y_n\rangle_\mathcal{K}\right|\\
& =  \left|\langle \rho U_T y_n ,\rho U_A y_n\rangle_\mathcal{K} - \langle T y_n, \rho U_A y_n\rangle_\mathcal{K}\right|\\
& = \left|\langle \rho U_T y_n-T y_n, \rho U_A y_n\rangle_\mathcal{K} \right|\\
& \leq \left\| \rho U_T y_n-T y_n\right\|\rho.
\end{align*}
It follows from (\ref{Eq:1}) that
\[\lim_{n\to \infty} \left|\langle \rho U_T y_n, \rho U_A y_n\rangle_\mathcal{K} - \langle T y_n, A y_n\rangle_\mathcal{H} \right|  = 0.\]
Thus, we obtained a sequence of unit vectors $(y_n)$ in $\mathcal{K}$ such that
\[
\lim_{n \to \infty}\|\rho U_Ty_n\|=\rho=\|\rho U_T\| \quad \mathrm{and}\quad \lim_{n\to \infty}\langle \rho U_Ty_n, \rho U_Ay_n \rangle_\mathcal{K}=0.
\]
Therefore, by Theorem \ref{Bhatia-Semrl} we have $\rho U_T\perp_B \rho U_A$. Again, the homogeneity of Birkhoff-James orthogonality gives $U_T\perp_B U_A$. This completes the proof.
\end{proof}

A unitary dilation of a contraction is a special case of unitary $\rho$-dilation when $\rho =1$. So, we have the following corollaries.

\begin{cor}\label{Corollary for isometric dilation}
Let $T$ and $A$ be two contractions acting on a Hilbert space $\mathcal{H}$ with $\|T\|=1$. If $T$ and $A$ dilate to unitaries (or isometries) $U_T, U_A$ respectively on a common space $\mathcal K \supseteq \mathcal H $, then $T\perp_B A$ implies that $U_T\perp_B U_A$.
\end{cor}

\begin{proof}
The `unitary dilation' part of this corollary is a special case of Theorem \ref{Orthogonality of rho dilations} with $\rho=1$. The `isometric dilation' part follows directly by replacing $\rho = 1$, $U_T=V_T$ and $U_A=V_A$ in Theorem \ref{Orthogonality of rho dilations}.
\end{proof}

\begin{cor}
Let $T$ and $A$ be two contractions acting on a Hilbert space $\mathcal{H}$ such that $T$ is self-adjoint with $\|T\|=1$. If $T$ and $A$ admit unitary (or isometric) dilations $U_T$ and $U_A$ respectively on a common space $\mathcal K \supseteq \mathcal H$, then $T\perp_B A$ and $T^2\perp_B A$ imply that $U_T^k\perp_B U_A$ for all positive integers $k$.

\end{cor}

\begin{proof}
We only prove the `unitary dilation' part and the proof for the `isometric dilation' part is similar. Since $T\perp_B A$ and $T^2\perp_B A$, it follows from Corollary \ref{Corollary about self-adjoint} that $T^k\perp_B A$ for all $k\geq 0$. For any unitary dilation $U_T$ of $T$ on $\mathcal{K}$, $U_T^{k}$ is evidently a unitary dilation of $T^k$ on $\mathcal{K}$. It follows from Corollary \ref{norming sequence} that $\|T^k\|=1$, for all $k\geq 1$. Therefore, by Corollary \ref{Corollary for isometric dilation}, we have that $U_T^k\perp_B U_A$ for all natural numbers $k$ as desired.
\end{proof}

\begin{cor}
Let $\mathcal{H}$ be a Hilbert space and $T$ be a contraction on $\mathcal{H}$ such that $\|D_T\|=1$. Then the following hold.
\begin{itemize}
    \item[(i)] Suppose $T$ and $D_T$ dilate to unitaries on a common space $\mathcal{K}\supseteq \mathcal{H}$. Then $U_{D_T}^k\perp_B U_T^j$ for any unitary dilations $U_{D_T}$ of $D_T$ and $U_T$ of $T$ on $\mathcal{K}$, where $k,j$ are any positive integers. 

    \item[(ii)] Suppose $T$ and $D_T$ dilate to isometries on a common space $\mathcal{K}\supseteq \mathcal{H}$. Then $V_{D_T}^k\perp_B V_T^j$ for any isometric dilations $V_{D_T}$ of $D_T$ and $V_T$ of $T$ on $\mathcal{K}$, where $k,j$ are any positive integers.
\end{itemize}
\end{cor}

\begin{proof}

We prove (i) only as proof for (ii) is similar. Let $U_{D_T}$ and $U_T$ be any unitary dilations of $D_T$ and $T$ respectively on $\mathcal{K}$. Therefore, for any positive integers $k,j$, the unitaries $U_{D_T}^k$ and $U_T^j$ are unitary dilations of $D_T^k$ and $T^j$ respectively. By Corollary \ref{Orthogonality of T and DT}, we have $D_T^k\perp_B T^j$. Since $\|D_T^k\|=1$, it follows from Theorem \ref{Orthogonality of rho dilations} that $U_{D_T}^k\perp_B U_T^j$.

\end{proof}

The assumption that $\|T\|=\rho$ in Theorem \ref{Orthogonality of rho dilations} is crucial and cannot be ignored. The following example shows that the orthogonality in $\zeta_\rho$ class does not carry over in general to the class of corresponding unitary $\rho$-dilations unless $\|T\|=\rho$.

\begin{example}\label{Example: I}
Consider the following $2\times 2$ scalar matrices
\[
T =\frac{1}{\sqrt{2}}\begin{bmatrix}
1 & 0\\
0 & 1
\end{bmatrix}, \qquad A =\frac{1}{\sqrt{2}}\begin{bmatrix}
0 & -1\\
1 & 0
\end{bmatrix}.
\]
Evidently, $T$ and $A$ are members of $\zeta_1$ and $\|T\|=\|A\|=\dfrac{1}{\sqrt{2}}$. The defect spaces of both $T$ and $A$ are equal to $\mathcal{H}$. Thus, both $T$ and $A$ admit unitary dilations over the same minimal space $\mathcal K=\bigoplus_{-\infty}^\infty \mathcal{H}$. Also, $T$ and $A$ are mutually orthogonal as $T$ and $A$ are scalar multiple of unitary operators and $\langle T(1,0), A(1,0)\rangle=0$.
Let $\widetilde{U_T}$ and $\widetilde{U_A}$ as in Equation-(\ref{eqn:uni-dil}) be the Sch\"{a}ffer dilations of $T$ and $A$ respectively on $\mathcal{K}$. Then for any unit vector $\mathbf{x}=(\dots, x_{-1}, x_0, x_1, \dots)$ in $\mathcal{K}$ we have
\[\left\langle \widetilde{U_T}\mathbf{x}, \widetilde{U_A}\mathbf{x}\right\rangle_\mathcal{K}=\langle \mathsf{T}_U(x_0,x_1), \mathsf{A}_U (x_0,x_1)\rangle_{\mathcal{H}\oplus \mathcal{H}}\; +1-(\|x_0\|^2+\|x_1\|^2),\]
where $\mathsf{T}_U$ and $\mathsf{A}_U$ are the $2 \times 2$ Halmos blocks as in (\ref{Halmos block}), i.e.
\[
\mathsf{T}_U=\begin{bmatrix}
D_T & -T^*\\
T & D_{T^*}
\end{bmatrix} \qquad \mathrm{and} \qquad \mathsf{A}_U=\begin{bmatrix}
D_A & -A^*\\
A & D_{A^*}
\end{bmatrix}.
\]
Evidently
\[
T=T^*=D_T=D_{T^*}=D_A=D_{A^*} \quad \mathrm{and} \quad A^*=-A.
\]
Now,
\begin{align*}
\left\langle \mathsf{T}_U(x_0,x_1), \mathsf{A}_U(x_0,x_1)\right\rangle & = \left\langle (D_Tx_0-T^*x_1, Tx_0+D_{T^*}x_1), (D_Ax_0-A^*x_1, Ax_0+D_{A^*}x_1)\right\rangle\\
& = \left\langle (Tx_0-Tx_1, Tx_0+Tx_1), (Tx_0+Ax_1, Ax_0+Tx_1)\right\rangle\\
& = \frac{1}{2}(\|x_0\|^2+\|x_1\|^2)+i~\mathfrak{Im}\langle x_0, x_1\rangle+i\sqrt{2}~\mathfrak{Im}\langle x_0, Ax_1 \rangle\\
&\qquad +\frac{1}{\sqrt{2}}(\langle x_0,Ax_0\rangle-\langle x_1,Ax_1\rangle).
\end{align*}
Therefore, 
\begin{align*}
\left\langle \widetilde{U_T}\mathbf{x}, \widetilde{U_A}\mathbf{x}\right\rangle_\mathcal{K} & =1-\frac{1}{2}(\|x_0\|^2+\|x_1\|^2)+i~\mathfrak{Im}\langle x_0, x_1\rangle+i\sqrt{2}~\mathfrak{Im}\langle x_0, Ax_1 \rangle\\
&\qquad +\frac{1}{\sqrt{2}}(\langle x_0,Ax_0\rangle-\langle x_1,Ax_1\rangle)
\end{align*}
Note that $\frac{1}{\sqrt{2}}(\langle x_0,Ax_0\rangle-\langle x_1,Ax_1\rangle)$ is either zero or purely imaginary. Thus,
\[\mathfrak{Re}\left\langle \widetilde{U_T}\mathbf{x}, \widetilde{U_A}\mathbf{x}\right\rangle_\mathcal{K} \geq \frac{1}{2},\]
for all unit vectors $\mathbf{x}$ in $\mathcal{K}$. Consequently, $\widetilde{U_T}\not\perp_B \widetilde{U_A}$. \qed

\end{example}

\medskip

The norm attainment set plays an important role in Birkhoff-James orthogonality. Let $M_T$ be the norm attainment set of an operator $T$ acting on a Hilbert space $\mathcal{H}$, i.e.
\[
M_T=\{x\in S_{\mathcal{H}}:~\|Tx\|=\|T\|\}.
\]
We learn from the literature that $M_T=S_{\mathcal{H}_0}=\{ x\in \mathcal H_0\,:\, \|x\|=1 \}$ for some subspace $\mathcal{H}_0$ of $\mathcal{H}$ (e.g. see \cite{Sain & Paul}). Now Equation-(\ref{Eq:1}) shows that if $\|T\|=\rho$ and $M_T\neq \emptyset$, then for any $x\in M_T$, $Tx=\rho U_Tx$, for all unitary $\rho$-dilations $U_T$ of $T$. Therefore, $\rho U_T$ is an extension of $T$ on $\mathrm{span}\{M_T\}=\mathcal{H}_0$. In particular, this is true for any compact operator which is a member of $\zeta_\rho$.\\

We now show that the converse of Theorem \ref{Orthogonality of rho dilations} is not always true, i.e. unitary $\rho$-dilations $U_T$ of $T$ and $U_A$ of $A$ can be Birkhoff-James orthogonal even if $T\not\perp_B A$. This is apparent from the right-symmetry of the unitary operators. The following example shows that counter examples can be constructed in individual $\zeta_\rho$ classes. Indeed, we can find contractions $T, A$ with $T \not\perp_B A$ such that any unitary $\rho$-dilations $U_T$ of $T$ and $U_A$ of $A$ acting on a common space $\mathcal K \supseteq \mathcal H$ are orthogonal. 

\begin{example}\label{Example for the converse}

Choose any $\rho\in (0,\infty)$ and let us consider the following two $4\times 4$ scalar matrices
\[
A = \begin{bmatrix}
0        & 0 & 0 & 0\\
\rho     & 0 & 0 & 0\\
0        & 0 & 0 & \rho\\
0        & 0 & 0 & 0
\end{bmatrix},\qquad
T=\begin{bmatrix}
0        & 0 & 0 & 0\\
\rho     & 0 & 0 & 0\\
0        & 0 & 0 & 0\\
0        & 0 & 0 & 0
\end{bmatrix}.
\]
It is easy to see that $T^n=A^n=\mathbf{0}$ for all $n\geq 2$ and $\|T\|=\|A\|=\rho$. Therefore, $T,A\notin \zeta_{\rho'}$ for any $\rho'$ satisfying $0<\rho'<\rho< \infty$. We now show that $T, A\in \zeta_\rho$. Let $\mathcal{K}$ be a Hilbert space with an orthonormal basis $\{\phi_m:~m\in \mathbb{Z}\}$. Let $\{e_1,e_2,e_3,e_4\}$  denote the standard orthonormal basis of $\mathbb{C}^4$. Identify $e_j$ to $\phi_j$ for each $j\in \{1,2,3,4\}$. Consider bijective maps $f: \mathbb{Z}\to \mathbb{Z}$ and $g:\mathbb{Z}\to \mathbb{Z}$ defined by
\begin{align*}
& f(m)=m+2, \qquad  m\geq 5,~ m\leq -2,\\
& f(-1)=4, f(0)=1, f(1)=2, f(2)=5, f(3)=6, f(4)=3;\\
& g(m)=m+3, \qquad m\geq 2,~ m\leq -3;\\
& g(-2)=4,~g(-1)=3,~ g(0)=1,~g(1)=2.
\end{align*}
Notice that
\begin{align*}
f^n\{1,2,3,4\}\cap \{1,2,3,4\}=\emptyset, \qquad n\geq 2,\\
g^n\{1,2,3,4\}\cap \{1,2,3,4\}=\emptyset, \qquad n\geq 2.
\end{align*}
Consider the linear operators $U_A$ and $U_T$ on $\mathcal{K}$ defined by
\begin{align*}
& U_A\phi_m=\phi_{f(m)}, \qquad m\in \mathbb{Z}\\
& U_T\phi_m=\phi_{g(m)}, \qquad m\in \mathbb{Z}.
\end{align*}
Then $U_A,U_T$ are unitary operators on $\mathcal{K}$, and we have \begin{center}
\begin{tabular}{l l}
 $\rho P_\mathcal{H}U_A e_1=\rho P_\mathcal{H}e_2=\rho e_2,$ $\quad$  &  $\rho P_\mathcal{H}U_T e_1=\rho P_\mathcal{H}e_2=\rho e_2,$\\
$\rho P_\mathcal{H}U_A e_2=\rho P_\mathcal{H}\phi_5=\mathbf{0},$ $\quad$  &  $\rho P_\mathcal{H}U_T e_2=\rho P_\mathcal{H}\phi_5=\mathbf{0},$\\
 $\rho P_\mathcal{H}U_A e_3=\rho P_\mathcal{H}\phi_6=\mathbf{0}$, $\quad$  &  $\rho P_\mathcal{H}U_T e_3=\rho P_\mathcal{H}\phi_6=\mathbf{0},$\\
 $\rho P_\mathcal{H}U_A e_4=\rho P_\mathcal{H}e_3=\rho e_3,$ $\quad$  &  $\rho P_\mathcal{H}U_T e_4=\rho P_\mathcal{H}\phi_7=\mathbf{0}.$
\end{tabular}
\end{center}

\noindent Moreover, for any $n\geq 2$ we have
\begin{align*}
\rho P_\mathcal{H}U_A^n e_m=\rho P_\mathcal{H} \phi_{f^n(m)}=\mathbf{0},& \qquad m\in \{1,2,3,4\},\\
\rho P_\mathcal{H}U_T^n e_m=\rho P_\mathcal{H} \phi_{g^n(m)}=\mathbf{0},& \qquad m\in \{1,2,3,4\}.
\end{align*}
Therefore,
\begin{align*}
\rho P_\mathcal{H}U_A^n h = A^n h, \qquad n\geq 1, \quad h\in \{e_1,e_2,e_3,e_4\},\\
\rho P_\mathcal{H}U_T^n h = T^n h, \qquad n\geq 1, \quad h\in \{e_1,e_2,e_3,e_4\}.
\end{align*}
Consequently, the same relations hold for arbitrary members of $\mathbb{C}^4$. Thus, $U_T$ and $U_A$ on $\mathcal{K}$ are unitary $\rho$-dilation of $T$ and $A$ respectively and $T, A\in \zeta_\rho$.

\medskip

Evidently, $A\perp_B T$, as $\|Ae_4\|=\rho=\|A\|$ and $\langle Ae_4,Te_4\rangle = 0$. Let $\mathbf{U}_A$ and $\mathbf{U}_T$ on a Hilbert space $\mathcal{K}$ be any unitary $\rho$-dilations of $T$ and $A$ respectively. Since $\|A\|=\rho$, by Theorem \ref{Orthogonality of rho dilations}, we have $\mathbf{U}_A\perp_B \mathbf{U}_T$. By the right-symmetry of $\mathbf{U}_T$, we also have $\mathbf{U}_T\perp_B \mathbf{U}_A$. However, $T \not\perp_B A$ as $M_T=\{\mu e_1:~|\mu|=1\}$ and $\langle Te_1, Ae_1\rangle =\rho^2$. \qed

\end{example}

\vspace{0.2cm}

\section{Dilation and approximate Birkhoff-James orthogonality} \label{sec:06}

\vspace{0.4cm}

\noindent Recall that for any two elements $x,y$ in a normed space $\mathbb X$ and for any $\varepsilon\in [0,1)$, $x$ is said to be $\varepsilon$-\textit{approximate Birkhoff-James orthogonal} to $y$, denoted by $x\perp_B^\varepsilon y$, if 
\begin{equation*}
\|x+\lambda y\|^2\geq \|x\|^2-2\varepsilon\|x\|\|\lambda y\|, \qquad \mathrm{for~all~scalars~}\lambda.    
\end{equation*}
It is evident that approximate orthogonality is homogeneous and for $\varepsilon=0$ it coincides with Birkhoff-James orthogonality. Moreover, if $x\perp_B^{\varepsilon_0} y$ for some $\varepsilon_0\in [0,1)$, then $x\perp_B^{\varepsilon} y$, for all $\epsilon \in (\varepsilon_0,1)$. A characterization of $\varepsilon$-approximate orthogonality of complex matrices was found in Theorem 4.2 \& Theorem 4.6 in \cite{Roy & Sain II}. Here we generalize this for operators on complex Hilbert spaces and this is one of the main results.

\begin{theorem}\label{Basic: II}
Let $\mathcal H$ be a Hilbert space and let $T,A\in \mathcal{B}(\mathcal{H})$. For any $\varepsilon\in [0,1)$, $T\perp_B^\varepsilon A$ if and only if there is a sequence of unit vectors $(x_n)_{n\in \mathbb{N}}$ in $\mathcal{H}$ such that 
\[
\lim_{n \to \infty} \|Tx_n\|= \|T\| \quad \mathrm{and} \quad \underset{n \to \infty}{\lim}|\langle Tx_n, Ax_n\rangle|< \varepsilon\|T\|\|A\|.\]
\end{theorem}

\begin{proof}
Since $T\perp_B^\varepsilon A$, it follows from Theorem 2.1 in \cite{Woijcik} (which gives a characterization of approximate orthogonality in a Banach space) that there exists $S\in \mathrm{span}\{T,A\}$ such that 
\[
T\perp_B S\quad \mathrm{and} \quad \|A-S\|\leq \varepsilon \|A\|.
\]
Consequently, by Theorem \ref{Bhatia-Semrl}, we can find a norming sequence $(x_n)$ of $T$ such that
\[
\lim_{n\to \infty} \langle Tx_n, Sx_n\rangle =0.\]
Without loss of generality let us assume that the sequence $\left(|\langle Tx_n, Ax_n\rangle|\right)_{n\in \mathbb{N}}$ is convergent, because, otherwise we will consider a suitable subsequence $(x_{n_j})$ of $(x_n)$ for which $\left(|\langle Tx_{n_j}, Ax_{n_j}\rangle|\right)_{j\in \mathbb{N}}$ is convergent. Now,
\begin{align*}
|\langle Tx_n, Ax_n\rangle| & \leq |\langle Tx_n, (A-S)x_n\rangle|+|\langle Tx_n, Sx_n\rangle|\\
& \leq \|T\|\|(A-S)\|+|\langle Tx_n, Sx_n\rangle|.
\end{align*}
Taking limits on the both sides we have
\[
\lim_{n\to \infty} |\langle Tx_n, Ax_n\rangle|\leq \varepsilon\|T\|\|A\|.
\]
Conversely, suppose that there is a sequence of unit vectors $(x_n)$ in $\mathcal{H}$ satisfying the stated conditions. Again, without loss of generality we may assume that the sequence $(|\langle Tx_n, Ax_n \rangle|)_{n\in \mathbb{N}}$ is convergent. Thus, for any scalar $\lambda$ we have
\begin{align*}
\|T+\lambda A\|^2 & \geq \|Tx_n+\lambda Ax_n\|^2\\
& = \|Tx_n\|^2 + |\lambda|\|Ax_n\|^2 + 2~\mathfrak{Re}~\lambda \langle Tx_n, Ax_n \rangle\\
& \geq \|Tx_n\|^2 -2|\lambda||\langle Tx_n, Ax_n \rangle|\\
& \geq \lim_n \|Tx_n\|^2 - 2|\lambda|\lim_n|\langle Tx_n, Ax_n \rangle|\\
& \geq \|T\|^2 - 2 \varepsilon \|T\| \|\lambda A\|.
\end{align*}
This is the desired inequality (\ref{Definition of Approx. orthogonality}). This completes the proof.
\end{proof}

Now we are going to show that if $T \perp_B A$ for two operators $T,A$, then any unitary $\rho$-dilations $U_T$ of $T$ and $U_A$ of $A$ acting on a common space are $\varepsilon$-approximate Birkhoff-James orthogonal with a proper bound on $\varepsilon$. Also, we show that the bound is sharp. This is one of the main results of this article.

\begin{theorem}\label{Approximate orthogonality of dilation}
Let $T,A$ be two operators on a Hilbert space $\mathcal{H}$ and $\rho > 0$. Suppose that $T$ and $A$ are members of class $\zeta_\rho$ and possess unitary $\rho$-dilations on a common space $\mathcal{K}\supseteq \mathcal{H}$. Then $T\perp_B A$ implies that $U_T\perp_B^\varepsilon U_A$, for all unitary $\rho$-dilations $U_T$ of $T$ and $U_A$ of $A$ on $\mathcal{K}$ and for every $\varepsilon\in [\kappa,1)$, where
\begin{equation}
\kappa=\eta_0\left(\sqrt{1-\frac{\|T\|^2}{\rho^2}}\right),
\end{equation}
\[\eta_0=\min\{\eta_1,\eta_2\},\]
and
\begin{equation}\label{Bound I}
\eta_1=\sup\left\{\zeta\geq 0:~\lim_n\sqrt{1-\frac{\|Ax_n\|^2}{\rho^2}} = \zeta,~~\|x_n\|=1,~\lim_n\|Tx_n\| = \|T\|\right\},   
\end{equation}
\begin{equation}\label{Bound II}
\eta_2=\sup\left\{\zeta\geq 0:~\lim_n\sqrt{1-\frac{\|A^*x_n\|^2}{\rho^2}} = \zeta,~\|x_n\|=1,~\lim_n\|T^*x_n\| = \|T^*\|\right\}.   
\end{equation}
Moreover, $U_T\perp_B U_A$ if $\|T\|=\rho$. In addition, if $\|A\|=\rho$ and if every norming sequence of $T$ $($or $T^*)$ is also a norming sequence of $A ($or $A^*)$, then $U_T\perp_B U_A$. 

\end{theorem}

\begin{proof}
As usual, since $T\perp_B A$, there exists a norming sequence $(x_n)_{n\in \mathbb{N}}$ of $T$ such that 
\[
\lim_{n\to \infty } \langle Tx_n, Ax_n\rangle_\mathcal{H} =0.
\] 
We now complete the proof in the following two steps.

\medskip

\noindent \emph{Step I}: Note that
\begin{align*}
\left|\langle U_T x_n, U_A x_n\rangle_{\mathcal{K}}\right| & =  \left|\left\langle P_{\mathcal{H}}U_T x_n + P_{\mathcal{H}^\perp}U_T x_n, P_{\mathcal{H}}U_Ax_n+P_{\mathcal{H}^\perp}U_A x_n \right\rangle_{\mathcal{K}}\right|\\
& \leq  \left|\left\langle P_{\mathcal{H}}U_T x_n, P_{\mathcal{H}}U_Ax_n\right\rangle_\mathcal{K}\right|+\left|\left\langle P_{\mathcal{H}^\perp}U_T x_n, P_{\mathcal{H}^\perp}U_Ax_n\right\rangle_\mathcal{K}\right|\\
& \leq \frac{1}{\rho^2}\left|\left\langle T x_n, Ax_n\right\rangle_\mathcal{H}\right| + \left\|P_{\mathcal{H}^\perp}U_T x_n\right\|\left\|P_{\mathcal{H}^\perp}U_A x_n\right\|\\
& = \frac{1}{\rho^2}\left|\left\langle T x_n, Ax_n\right\rangle_\mathcal{H}\right| + \left\|U_Tx_n -P_{\mathcal{H}}U_T x_n\right\|\left\|U_A x_n-P_{\mathcal{H}}U_A x_n\right\|\\
& = \frac{1}{\rho^2}\left|\left\langle T x_n, Ax_n\right\rangle_\mathcal{H}\right| + \left\|U_Tx_n -\frac{1}{\rho}T x_n\right\|\left\|U_A x_n-\frac{1}{\rho}A x_n\right\|\\
& =\frac{1}{\rho^2} \left|\left\langle T x_n, Ax_n\right\rangle_\mathcal{H}\right| + \sqrt{1-\frac{\|Tx_n\|^2}{\rho^2}}\sqrt{1-\frac{\|Ax_n\|^2}{\rho^2}},
\end{align*}
where the last equality follows from the fact that
\begin{align*}
\left\|U_T x_n-\frac{1}{\rho}T x_n\right\|^2 & = 1-\frac{2}{\rho}~\mathfrak{Re}~\langle U_Tx_n, T x_n \rangle_\mathcal{K}+\frac{1}{\rho^2}\|Tx_n\|^2\\
& = 1-\frac{2}{\rho}~\mathfrak{Re}~\langle P_\mathcal{H}U_Tx_n, T x_n \rangle_\mathcal{K}+\frac{1}{\rho^2}\|Tx_n\|^2\\
& = 1-\frac{2}{\rho^2}\|Tx_n\|^2+\frac{1}{\rho^2}\|Tx_n\|^2\\
& = 1-\frac{1}{\rho^2}\|Tx_n\|^2,
\end{align*}
and similarly,
\[\left\|U_A x_n-\frac{1}{\rho}A x_n\right\|^2= 1-\frac{1}{\rho^2}\|Ax_n\|^2.\]
Passing through a suitable subsequence, if necessary, we may assume that the two sequences $\left(\left|\langle U_T x_n, U_A x_n\rangle_{\mathcal{K}}\right|\right)$ and $\left(\sqrt{1-\frac{\|Ax_n\|^2}{\rho^2}}\right)$ are convergent. Consequently,
\begin{align*}
\lim_{n\to \infty }\left|\langle U_T x_n, U_A x_n\rangle_{\mathcal{K}}\right| & \leq \lim_{n\to \infty} \left|\left\langle T x_n, Ax_n\right\rangle_\mathcal{H}\right| + \lim_{n\to \infty}\sqrt{1-\frac{\|Tx_n\|^2}{\rho^2}}\sqrt{1-\frac{\|Ax_n\|^2}{\rho^2}}\\
& \leq \eta_1\sqrt{1-\frac{\|T\|^2}{\rho^2}}, \qquad [\mathrm{by} ~(\ref{Bound I})].
\end{align*}
Since $\|U_Tx_n\|=1$ for every natural number $n$, it follows from Theorem \ref{Basic: II} that 
\[U_T\perp_B^{\varepsilon'} U_A, \quad \mathrm{for~every} \; \;  \varepsilon'  \text{ satisfying } \; ~\eta_1\sqrt{1-\frac{\|T\|^2}{\rho^2}}\leq \varepsilon' < 1.\]

\medskip

\noindent \emph{Step II:} Note that the class $\zeta_\rho$ is $\ast$-closed and $U_T^*$ and $U_A^*$ are unitary $\rho$-dilations of $T^*$ and $A^*$ on $\mathcal{K}$, respectively. Since $T\perp_B A$, we also have $T^*\perp_B A^*$. Thus, there exists a norming sequence $(y_n)$ of $T^*$ such that $\lim_n\langle T^* y_n, A^*y_n\rangle_\mathcal{H}=0.$ Now, proceeding similarly as in \emph{Step I} and using (\ref{Bound II}), we have that
\[U_T^*\perp_B^{{\varepsilon''}} U_A^*, \quad \mathrm{for~every~} \varepsilon'' \text{ satisfying }\; \eta_2\sqrt{1-\frac{\|T\|^2}{\rho^2}}\leq \varepsilon'' <1.\]
Consequently, for any scalar $\lambda$
\[\|U_T+\lambda U_A\|^2  = \|U_T^*+\overline{\lambda} U_A^*\|^2\geq \|U_T^*\|^2-2{\varepsilon''} \|U_T^*\|\|\lambda U_A^*\|=\|U_T\|^2-2{\varepsilon''} \|U_T\|\|\lambda U_A\|.\]
This is same as saying that
\[
U_T\perp_B^{{\varepsilon''}} U_A, \quad \mathrm{for~every~} \varepsilon'' \text{ satisfying }\;  \eta_2\sqrt{1-\frac{\|T\|^2}{\rho^2}}\leq \varepsilon'' <1.
\]
Therefore, $U_T\perp_B^\kappa U_A$, where 
\[\kappa=\min\{\varepsilon', \varepsilon''\} = \eta_0\left(\sqrt{1-\frac{\|T\|^2}{\rho^2}}\right),
\]
and $\eta_0=\min\{\eta_1, \eta_2\}$. Consequently, $U_T\perp_B^\varepsilon U_A$ for every $\varepsilon\in [\kappa,1)$.

\medskip

If $\|T\|=\rho$, then $\sqrt{1-\frac{\|T\|^2}{\rho^2}} = 0$. Thus, $U_T\perp_B U_A$. In addition, if $\|A\|=\rho$ and every norming sequence of $T$ (or $T^*$) is also a norming sequence of $A$ (or $A^*$), then again $\eta_0=0$, and we have $U_T\perp_B U_A$. This completes the proof.

\end{proof}

The following is an obvious corollary of the preceding theorem which deals with $\varepsilon$-approximate Birkhoff-James orthogonality of isometric dilations of a pair of contractions.

\begin{cor}
Let $T$ and $A$ be two contractions acting on a Hilbert space $\mathcal{H}$. Suppose $T$ and $A$ admit isometric dilations $V_T$ and $V_A$ respectively on a common space $\mathcal{K}\supseteq \mathcal{H}$. Then $T\perp_B A$ implies that $V_T\perp_B^\varepsilon V_A$, for every $\varepsilon\in [\gamma,1)$, where 
\[\gamma = \delta_0\left(\sqrt{1-{\|T\|^2}}\right),\]
\[\delta_0=\sup \left\{\sigma\geq 0:~\sqrt{1-{\|Ax_n\|^2}}\to \sigma,~\|x_n\|=1,~\|Tx_n\|\to \|T\|\right\}.\]
Moreover, if $\|A\|=1$ and every norming sequence of $T$ is also a norming sequence of $A$, then $V_T\perp_B V_A$. 
\end{cor}

\begin{proof}
The proof follows by replacing $\rho=1$ in \emph{Step I} in the proof of Theorem \ref{Approximate orthogonality of dilation}.
\end{proof}

The bound that we obtained for $\varepsilon$ in Theorem \ref{Approximate orthogonality of dilation} is universally optimal. The following example illustrates this.

\begin{example}
Let $T$ and $A$ be the contractions as in Example \ref{Example: I}. Note that $T$ and $A$ are members of $\zeta_1$. Let $\widetilde{U_T}$ and $\widetilde{U_A}$ be the Sch\"{a}ffer dilations as in Equation-(\ref{eqn:uni-dil}) of $T$ and $A$ respectively on the space $\mathcal{K}=\bigoplus_{-\infty}^{\infty}\mathcal{H}$. Then from Example \ref{Example: I}, it is easy to see that
\begin{equation}\label{bound for epsilon}
\left|\left\langle \widetilde{U_T}\mathbf{x}, \widetilde{U_A}\mathbf{x}\right\rangle_\mathcal{K}\right| \geq \frac{1}{2}, \qquad \|\mathbf{x}\|=1,~\mathbf{x}\in \mathcal{K}.    
\end{equation}
However, the lower bound $\dfrac{1}{2}$ is attained, if we consider $\mathbf{x}=(\dots, {\bf 0},{\bf 0},{\bf 0}, \dots, (1,0), {\bf 0},{\bf 0},{\bf 0} \dots)$, where the vector $(1,0)$ is situated in the zeroth co-ordinate. Being consistent with the notations of Theorem \ref{Approximate orthogonality of dilation}, we have $\eta_0=\dfrac{1}{\sqrt{2}}$ and $\kappa=\frac{1}{2}$. Therefore, $\widetilde{U_T}\perp_B^\varepsilon \widetilde{U_A}$ for all $\varepsilon\in [\frac{1}{2}, 1)$, and by (\ref{bound for epsilon}), we have $\widetilde{U_T}\not\perp_B^\varepsilon \widetilde{U_A}$, for any $\varepsilon$ less than $ \frac{1}{2}$. \qed

\end{example}

\vspace{0.3cm} 

\section{Generalized Sch\"{a}ffer unitary dilations and Birkhoff-James orthogonality} \label{sec:07}

\vspace{0.4cm}

\noindent In Sections \ref{sec:03} \& \ref{sec:05}, we have seen that the Sch\"{a}ffer dilations as in Equation-(\ref{eqn:uni-dil}) for two orthogonal contractions $T, A$ may or may not be orthogonal. In this Section, we furnish two different explicit Sch\"{a}ffer-type unitary dilations $\mathbb U_T$ and $\widehat{U_T}$ for a contraction $T$. Interestingly, for any contractions $T,A$ we must have $\mathbb U_T \perp_B \mathbb U_A$ even if $T$ and $A$ are not orthogonal. On the other hand, $\widehat{U_T} \perp_B \widehat{U_A}$ when $T \perp_B A$, i.e. the dilations of the type $\widehat{U_T}$ are orthogonality preserver.

\subsection{The first construction} 

\begin{theorem}\label{Force and Force}
Let $T$ be a contraction acting on a Hilbert space $\mathcal{H}$. Let $U_1$ and $U_2$ be any two unitary operators on $\mathcal{H}$. Also, let $(Y_{-n})_{n\in \mathbb{N}}$ and $(X_n)_{n\in \mathbb{N}}$ be any two sequences of unitary operators on $\mathcal{H}$.  Then the operator $\mathbb{U}_T$ on ${ \displaystyle \mathcal{K}=\bigoplus_{-\infty}^\infty \mathcal{H}}$, defined by the block matrix 
\begin{equation*} 
	\mathbb{U}_T =\left[
\begin{array}{ c c c c|c|c c c c}
\bm{\ddots}&\vdots &\vdots&\vdots   &\vdots  &\vdots& \vdots&\vdots&\vdots\\
\cdots&0&\mathrm{Y_{-2}}&0  &0&  0&0&0&\cdots\\
\cdots&0&0&\mathrm{Y_{-1}}  &0&  0&0&0&\cdots\\
\cdots&0&0&0  &\mathrm{U_2D_T}&  \mathrm{-U_2T^*U_1}&0&0&\cdots\\ \hline

\cdots&0&0&0   &\mathrm{T}&   \mathrm{D_{T^*}U_1}&0&0&\cdots\\ \hline

\cdots&0&0&0   &0&  0& \mathrm{X_1}&0&\cdots\\
\cdots&0&0&0   &0&  0&0&\mathrm{X_2}&\cdots\\
\cdots&0&0&0  &0&   0& 0&0&\cdots\\
\vdots&\vdots&\vdots&\vdots&\vdots&\vdots&\vdots&\vdots&\bm{\ddots}\\
\end{array} \right],
	\end{equation*}
is a unitary dilation of $T$.

\end{theorem}

\begin{proof}

It is evident from the block matrix of $\mathbb U_T$ that
$
T^n= P_{\mathcal{H}}\mathbb{U}_T^n|_{\mathcal{H}}$ for all $n\in \mathbb{N}$ and consequently $\mathbb U_T$ dilates $T$. We now show that $\mathbb U_T$ is a unitary operator. Let $\mathbf{h},\mathbf{h}'\in \mathcal{K}$ be such that  $\mathbf{h}=(h_i)_{i\in \mathbb{Z}}$, $\mathbf{h}'=(h_i')_{i\in \mathbb{Z}}$ and $\mathbb{U}_T \mathbf{h}=\mathbf{h}'$. For any $i,j \in \mathbb Z$, suppose $\mathbb U_T(i,j)$ denotes the $(i,j)^{th}$ entry in the block matrix of $\mathbb U_T$. Then we have
\[
h_i'=\sum_{j} \mathbb{U}_T(i,j)h_j, \quad i\in \mathbb{Z}.
\]
Thus
\[
h'_i=\begin{cases}
Y_{i+1}h_{i+1}, & \mathrm{if}~i\leq -2,\\
U_2D_Th_0-U_2T^*U_1h_1, & \mathrm{if}~i=-1,\\
Th_0+D_{T^*}U_1h_1, & \mathrm{if}~i=0,\\
X_ih_{i+1}, & \mathrm{if}~i\geq 1.
\end{cases}
\]
Since $X_i$ and $Y_i$ are unitary operators, we have
\[
\sum\limits_{i\neq 0,-1}\|h_i'\|^2 =\sum\limits_{i\neq 0,1}\|h_i\|^2.
\]
Also, using the fact that $U_1,U_2$ are unitaries and $TD_T=D_{T^*}T$, a straightforward computation reveals that the $2\times 2$ Halmos block
\[
\begin{bmatrix}
    U_2D_T & -U_2T^*U_1\\
    T   & D_{T^*}U_1
\end{bmatrix}
\]
is a unitary operator. Therefore, $\|h_{-1}'\|^2+\|h_0'\|^2 =\|h_0\|^2+\|h_1\|^2$. This shows that
\[
\|\mathbf{h}'\|^2=\langle\mathbf{h}, \mathbf{h}\rangle_\mathcal{K}=\sum_i\sum_{j} \|\mathbb{U}_T(i,j)h_j\|^2=\langle U_T\mathbf{h}, U_T\mathbf{h}\rangle_\mathcal{K}.
\]
Thus, $\mathbb{U}_T$ is an isometry. It remains to show that $\mathbb{U}_T$ is surjective. Consider any $\mathbf{h}=(h_i)_{i\in \mathbb{Z}}$ in $\mathcal{K}$. Then it is straight-forward that $\mathbf{h}''=(h_i'')_{i\in \mathbb{Z}}$ is a preimage of $\mathbf{h}$ under $\mathbb{U}_T$, where
\[
h_i''=\begin{cases}
Y_i^*h_{i-1}, & \mathrm{if}~i\leq -1,\\
D_TU_2^* h_{-1}+T^*h_0, & \mathrm{if}~i=0,\\
-U_1^*TU_2^*h_{-1}+U_1^*D_{T^*}h_0, & \mathrm{if}~i=1,\\
X_{i-1}^*h_{i-1}, & \mathrm{if}~i\geq 2.
\end{cases}
\]
Hence $\mathbb{U}_T$ is a unitary dilation of $T$ on $\mathcal{K}$ and the proof is complete.

\end{proof}

For a pair of contractions $T,A$, we now choose the sequences $(Y_{-n}),\; (X_n)$ in such a way that $\mathbb U_T$ becomes orthogonal with $\mathbb U_A$. Let $(W_{-n})_{n\in \mathbb N}$, $(Y_{-n})_{n\in \mathbb N}$, $(X_n)_{n\in \mathbb N}$, $(Z_n)_{n\in \mathbb N}$ be sequences of unitary operators on $\mathcal{H}$. For any fixed negative integer $k_0$, we choose $W_{k_0}$ and $Y_{k_0}$ such that $W_{k_0}\perp_B Y_{k_0}$. Such a choice can always be ensured. For example, given any orthonormal basis $\{e_n:~n\in \mathbb{N}\}$ of $\mathcal{H}$, choose $W_{k_0}=I_\mathcal{H}$ and define $Y_{k_0}$ in the following way:
\[
Y_{k_0}e_i=\begin{cases}
e_2, & \mathrm{if} ~ i=1,\\
e_1, & \mathrm{if} ~ i=2,\\
e_i, & \mathrm{otherwise}.
\end{cases}
\]
Then $Y_{k_0}\perp_B W_{k_0}$, since $\langle Y_{k_0}e_1, W_{k_0}e_1\rangle_\mathcal{H}=0$. Consider the unitary dilations $\mathbb{U}_T$ and $\mathbb{U}_A$, as in Theorem \ref{Force and Force}, where 
\begin{align*}
\mathbb{U}_T(i,i+1)=W_{i+1},~\mathbb{U}_A(i,i+1)=Y_{i+1}, \quad \mathrm{for}~ i\leq -2,\\
\mathbb{U}_T(i,i+1)=X_i,~\mathbb{U}_A(i,i+1)=Z_i, \quad \mathrm{for}~  i\geq 1.
\end{align*}
Now, consider the unit vector $\mathbf{h}=(h_n)_{n\in \mathbb{Z}}$ in $\mathcal{K}$ such that $h_{k_0}=e_1$ and $h_n=\mathbf{0}$ for $n\neq k_0$. Then 
\[\langle \mathbb{U}_T \mathbf{h}, \mathbb{U}_A \mathbf{h} \rangle_{\mathcal{K}}=\langle Y_{k_0}e_1, W_{k_0}e_1\rangle_\mathcal{H}=0,
\]
and consequently $\mathbb{U}_T\perp_B \mathbb{U}_A$.

\subsection{The second construction}

\begin{theorem} \label{thm:cons2}
Let $T$ and $A$ be two contractions acting on a Hilbert space $\mathcal{H}$ such that $T \perp_B A$. Then there are unitary dilations $\widehat{U_T}$ of $T$ and $\widehat{U_A}$ of $A$ on some Hilbert space $\widehat{\mathcal{K}}\supseteq \mathcal{H}$ such that $\widehat{U_T}\perp_B \widehat{U_A}$.
\end{theorem}

\begin{proof}

For any two contractions $T,A$ acting on a Hilbert space $\mathcal H$, we can always find a common space $\mathcal K \supseteq \mathcal H$ such that both $T$ and $A$ admit unitary dilation on $\mathcal K$. For example, we can choose ${ \displaystyle \mathcal K= \oplus_{-\infty}^{\infty} \mathcal H }$ and $\widetilde{U_T}, \widetilde{U_A}$ as in Equation-(\ref{eqn:uni-dil}) to be unitary dilations of $T$ and $A$ respectively.
Suppose $\frac{T}{\|T\|}$ and $\frac{A}{\|A\|}$ dilate to unitaries $U_{\frac{T}{\|T\|}}$ and $U_{\frac{A}{\|A\|}}$ respectively on $\mathcal{K} \supseteq \mathcal H$. Set
$N_T=\|T\|U_{\frac{T}{\|T\|}}$ and $N_A=\|A\|U_{\frac{A}{\|A\|}}$. Then evidently $N_T$ and $N_A$ are normal contractions on $\mathcal K$ and
\[
P_\mathcal{H}N_T^n|_\mathcal{H}=\|T\|^nP_\mathcal{H}U_\frac{T}{\|T\|}^n|_\mathcal{H}=\|T\|^n \left(\frac{T}{\|T\|}\right)^n=T^n, \qquad n\in \mathbb{N}\cup \{0 \}.
\]
Therefore, $N_T,N_A$ are normal dilations of $T$ and $A$ respectively on $\mathcal{K}$. Note that $\|N_T\|=\|T\|$ and $\|N_A\|=\|A\|$. Consider the defect operators $D_{N_T}$ and $D_{N_A}$ of $N_T$ and $N_A$, respectively, i.e.
$D_{N_T} = (I_\mathcal{K}-N_T^*N_T)^\frac{1}{2}$ and $D_{N_A} = (I_\mathcal{K}-N_A^*N_A)^\frac{1}{2}$. Let $\widehat{\mathcal{K}}=\bigoplus_{-\infty}^\infty \mathcal{K}$. Set two operators $\widehat{U_T}$ and $\widehat{U_A}$ on $\widehat{\mathcal{K}}$ in the following way:

\begin{equation*} 
	\widehat{U_T} =\left[
\begin{array}{ c c c c|c|c c c c}
\bm{\ddots}&\vdots &\vdots&\vdots   &\vdots  &\vdots& \vdots&\vdots&\vdots\\
\cdots&0&\mathrm{I}&0  &0&  0&0&0&\cdots\\
\cdots&0&0&\mathrm{I}  &0&  0&0&0&\cdots\\
\cdots&0&0&0  &\mathrm{D_{N_T}}&  \mathrm{-N_T^*}&0&0&\cdots\\ \hline

\cdots&0&0&0   &\mathrm{N_T}&   \mathrm{D_{N_T^*}}&0&0&\cdots\\ \hline

\cdots&0&0&0   &0&  0& \mathrm{I}&0&\cdots\\
\cdots&0&0&0   &0&  0&0&\mathrm{I}&\cdots\\
\cdots&0&0&0  &0&   0& 0&0&\cdots\\
\vdots&\vdots&\vdots&\vdots&\vdots&\vdots&\vdots&\vdots&\bm{\ddots}\\
\end{array} \right],
	\end{equation*}

\medskip

 \begin{equation*} 
	\widehat{U_A} =\left[
\begin{array}{ c c c c|c|c c c c}
\bm{\ddots}&\vdots &\vdots&\vdots   &\vdots  &\vdots& \vdots&\vdots&\vdots\\
\cdots&0&\mathrm{I}&0  &0&  0&0&0&\cdots\\
\cdots&0&0&\mathrm{I}  &0&  0&0&0&\cdots\\
\cdots&0&0&0  &\mathrm{-D_{N_A}}&  \mathrm{N_A^*}&0&0&\cdots\\ \hline

\cdots&0&0&0   &\mathrm{N_A}&   \mathrm{D_{N_A^*}}&0&0&\cdots\\ \hline

\cdots&0&0&0   &0&  0& \mathrm{I}&0&\cdots\\
\cdots&0&0&0   &0&  0&0&\mathrm{I}&\cdots\\
\cdots&0&0&0  &0&   0& 0&0&\cdots\\
\vdots&\vdots&\vdots&\vdots&\vdots&\vdots&\vdots&\vdots&\bm{\ddots}\\
\end{array} \right].
	\end{equation*}
Evidently $\widehat{U_T}$ is the Sch\"{a}ffer unitary dilation of $N_T$ as in Equation-(\ref{eqn:uni-dil}). Also, considering the $2 \times 2$ Halmos block of $A$ we see that
\[
\begin{bmatrix}
-D_{N_A} & N_A^*\\
N_A  & D_{N_A^*}
\end{bmatrix}=\begin{bmatrix}
-I_{\mathcal{K}} & \mathbf{0}\\
\mathbf{0}  & I_\mathcal{K}
\end{bmatrix}\begin{bmatrix}
D_{N_A} & -N_A^*\\
N_A  & D_{N_A^*}
\end{bmatrix}.
\]
Thus, $\begin{bmatrix}
D_{N_A} & -N_A^*\\
N_A  & D_{N_A^*}
\end{bmatrix}$
is also a unitary and consequently $\widehat{U_A}$ is a unitary dilation of $A$.\\

We now show that  $\widehat{U_T}\perp_B \widehat{U_A}$. Evidently, by the homogeneity of Birkhoff-James orthogonality, we have $\frac{T}{\|T\|}\perp_B \frac{A}{\|A\|}$. It now follows from Theorem \ref{Orthogonality of rho dilations} that $U_{\frac{T}{\|T\|}}\perp_B U_{\frac{A}{\|A\|}}.$ Again by virtue of homogeneity, we have
$N_T\perp_B N_A$. Therefore, there is a norming sequence say $(x_k)_{k\in \mathbb{N}}$ such that
\[
\lim_{k\to \infty} \langle N_Tx_k, N_Ax_k\rangle =0.
\]
We set $(1-\|T\|^2)^\frac{1}{2}(1-\|A\|^2)^\frac{1}{2}=\beta$. Then $\beta\in [0,1)$. Now, consider two sequences of unit vectors $(\mathbf{h}_k)_{k\in \mathbb{N}}$ and $(\mathbf{h'}_k)_{k\in \mathbb{N}}$ in $\widehat{\mathcal{K}}$ such that 
\[\mathbf{h}_k=\left(y_n^{(k)}\right)_{n\in \mathbb{Z}} \quad \mathrm{and} \quad \mathbf{h'}_k=\left(z_n^{(k)}\right)_{n\in \mathbb{Z}}, \qquad k\in \mathbb{N}\]
 where
\[y^{(k)}_0=y^{(k)}_1=\mathbf{0};\qquad z_n^{(k)}=\begin{cases} \mathbf{0}, & \mathrm{if}~n\neq 0;\\
x_k, & \mathrm{if}~n= 0.
\end{cases}\]
It is easy to see that
$\mathbf{h}_k\perp \mathbf{h}'_k$ for every $ k\in \mathbb{N}$. Thus, $\left(\widetilde{\mathbf{h}_k}\right)_{k\in \mathbb{N}}$ is a sequence of unit vectors, where
\[\widetilde{\mathbf{h}_k}=\sqrt{\dfrac{\beta}{1+\beta}}\mathbf{h}_k+\dfrac{1}{\sqrt{1+\beta}}\mathbf{h'}_k, \qquad k\in \mathbb{N}.\]
Now,
\begin{align*}
\left\langle \widehat{U_T}\widetilde{\mathbf{h}_k}, \widehat{U_A}\widetilde{\mathbf{h}_k}\right\rangle_{\widehat{\mathcal{K}}} & = \frac{1}{1+\beta}\langle N_T x_k, N_A x_k\rangle_\mathcal{K} + \frac{1}{1+\beta}\langle D_{N_T} x_k,- D_{N_A} x_k\rangle_\mathcal{K} + \frac{\beta}{1+\beta}\sum\limits_{n\in \mathbb{Z}}\left\|y_n^{(k)}\right\|^2\\
& = \frac{1}{1+\beta}\langle N_T x_k, N_A x_k\rangle_\mathcal{K}-\frac{1}{1+\beta}\left\langle D_{N_A}D_{N_T}x_k,x_k\right\rangle_\mathcal{K}+\frac{\beta}{1+\beta}\\
& = \frac{1}{1+\beta}\langle N_T x_k, N_A x_k\rangle_\mathcal{K}-\frac{\beta}{1+\beta}\|x_k\|^2+\frac{\beta}{1+\beta}\\
& =\frac{1}{1+\beta} \langle N_T x_k, N_A x_k\rangle_\mathcal{K},
\end{align*}
where the second last equality follows from the fact that 
\begin{align*}
(I_\mathcal{K}-N_A^*N_A)^\frac{1}{2}(I_\mathcal{K}-N_T^*N_T)^\frac{1}{2} & =\left(I_\mathcal{K}-\|A\|^2U_{\frac{A}{\|A\|}}^*U_{\frac{A}{\|A\|}}\right)^\frac{1}{2}\left(I_\mathcal{K}-\|T\|^2U_{\frac{T}{\|T\|}}^*U_{\frac{T}{\|T\|}}\right)^\frac{1}{2}\\
& = (1-\|A\|^2)^\frac{1}{2}I_\mathcal{K}(1-\|T\|^2)^\frac{1}{2}I_\mathcal{K}\\
& = \beta I_\mathcal{K}.
\end{align*}
Consequently, we have
\[
\lim_{k\to \infty} \left\langle \widehat{U_T}\widetilde{\mathbf{h}_k}, \widehat{U_A}\widetilde{\mathbf{h}_k}\right\rangle_{\widehat{\mathcal{K}}}=\frac{1}{1+\beta}\lim_{k\to \infty} \langle N_T x_k, N_A x_k\rangle_\mathcal{K}=0
\]
and hence $\widehat{U_T}\perp_B \widehat{U_A}$. This finishes the proof.

\end{proof}

\section{And\^{o} dilation, regular dilation and Birkhoff-James orthogonality} \label{sec:04}

\vspace{0.3cm}

\noindent In this Section, we consider a pair of contractions $T, ST$, where $S$ is a unitary, and study their orthogonality when $S$ commutes with $T$. In general, studying orthogonality of And\^{o} dilation of a pair of commuting contractions on a Hilbert space $\mathcal H$ is difficult because of the involvement of a unitary $G$ (see below and also CH-I of \cite{NFBK}) acting on $\mathcal H \oplus \mathcal H \oplus \mathcal H \oplus \mathcal H$. For this reason, we choose a particular class of commuting contractions of the form $(T,ST)$ and show that their And\^{o} dilations are orthogonal. Also, we prove the existence of orthogonal regular unitary dilation of a pair of commuting orthogonal contractions. We begin with a simple lemma.

\begin{Lemma}\label{Simple Lemma}
Let $T$ be a contraction and $S$ be a unitary acting on a Hilbert space $\mathcal{H}$. Then $D_T=D_{ST}$. 
\end{Lemma}

\begin{proof}
Evidently $ST$ is a contraction and we have
\[
D_T=(I_\mathcal{H}-T^*T)^{\frac{1}{2}}=(I_\mathcal{H}-T^*S^*ST)^{\frac{1}{2}}=(I_\mathcal{H}-(ST)^*ST)^{\frac{1}{2}}=D_{ST}.
\]  
\end{proof}

In this context, let us recall the definition of isometric dilation of a pair of commuting contractions. Let $(T_1,T_2)$ be a pair of commuting contractions acting on a Hilbert space $\mathcal H$. A pair of commuting isometries $(V_1,V_2)$ on $\mathcal K \supseteq \mathcal H$ is said to be an \textit{isometric dilation} of $(T_1,T_2)$ if 
\[
T_1^{n_1}T_2^{n_2}=P_{\mathcal H} V_1^{n_1}V_2^{n_2}|_{\mathcal H} \quad \text{ for all }\, n_1,n_2 \in \mathbb N \cup \{0\}.
\]
Also, in a similar way a \textit{unitary dilation} of $(T_1,T_2)$ can be defined.

\begin{theorem}\label{Class of examples}
Let $T$ be a contraction acting on a Hilbert space $\mathcal{H}$. Let $S$ be a unitary on $\mathcal{H}$. Then the following hold:
\begin{itemize}
    \item[(i)] If there exists a sequence of unit vectors $(y_n)_{n\in \mathbb{N}}$ in $\mathcal{H}$ such that 
\[
\lim_{n\to \infty} (1-\|Ty_n\|^2+\langle Ty_n,STy_n\rangle) \leq 0,
\]
then $\widetilde{U_T}\perp_B \widetilde{U_{ST}}$, where $\widetilde{U_T}$ and $\widetilde{U_{ST}}$ on $\mathcal{K}=\bigoplus_{-\infty}^\infty \mathcal{H}$ are the Sch\"{a}ffer unitary dilations of $T$ and $ST$ respectively as in Equation-$(\ref{eqn:uni-dil})$.

\item[(ii)] If $S$ commutes with $T$ and $\mathcal{W}(S)$ contains a non-positive real number, then there exists a pair of unitary operators $(U_T, U_{ST})$ that dilates the pair $(T,ST)$ on some Hilbert space $\mathcal{K}\supseteq \mathcal{H}$ such that $U_T\perp_B U_{ST}$.
\end{itemize}
\end{theorem}

\begin{proof}
(i) Write $ST=A$ and consider the sequence of unit vectors $(y_n,\mathbf{0})$ in $\mathcal{H}\bigoplus \mathcal{H}$. By Lemma \ref{Simple Lemma} we have
\begin{align*} 
\langle \mathsf{U}_A^*\mathsf{U}_T  (y_n,\mathbf{0}),  (y_n,\mathbf{0})\rangle_{\mathcal{H}\oplus \mathcal{H}} & =\langle \mathsf{U}_T  (y_n,\mathbf{0}), \mathsf{U}_A (y_k,\mathbf{0})\rangle_{\mathcal{H}\oplus \mathcal{H}}\\
& = \langle D_T y_n, D_A y_n \rangle_\mathcal{H}+\langle Ty_n, Ay_n\rangle_\mathcal{H}.\\
& = \langle D_T y_n, D_T y_n \rangle_\mathcal{H}+\langle Ty_n, STy_n\rangle_\mathcal{H}.\\
& = 1-\|Ty_n\|^2+\langle Ty_n, STy_n\rangle_\mathcal{H}.
\end{align*}
By the hypothesis of the theorem, we have
\[
\lim_{n\to \infty} \langle \mathsf{U}_A^*\mathsf{U}_T  (y_n,\mathbf{0}),  (y_n,\mathbf{0})_{\rangle\mathcal{H}\oplus \mathcal{H}}\leq 0.\]
Thus, by Corollary \ref{Corollary to Halmos Blocks}, we have $\widetilde{U_T}\perp_B \widetilde{U_A}$.

\medskip

(ii) It follows from the hypothesis that $\mathcal{W}(S)$ contains a non-positive real number, say $\beta$. Thus, there exists a sequence of unit vectors $(x_k)_{k\in\mathbb{N}}$ in $\mathcal{H}$ such that
\[
\lim_{k \to \infty}\langle Sx_k,x_k\rangle=\beta.\]
Since $S$ is normal and commutes with $T$, $S$ doubly commutes with $T$. Therefore,
\begin{align*}
SD_T^2=S(I_\mathcal{H}-T^*T)=S-ST^*T=S-T^*ST=S-T^*TS=(I_\mathcal{H}-T^*T)S=D_T^2S.  
\end{align*}
By iteration,
\[S(D_T^2)^n=(D_T^2)^nS,\qquad n\geq 0.\]
Therefore,
\[Sp(D_T^2)=p(D_T^2)S,\qquad \mathrm{for~every~polynomial~}p(z)\in \mathbb{C}[z].\]
We can find a sequence of polynomials $(p_n(z))_{n\in \mathbb{N}}$ in $\mathbb{C}[z]$ such that $(p_n(z))_{n\in \mathbb{N}}$ converges to $\sqrt{z}$ uniformly on the closed interval $[0,1]$. Thus, the sequence of operators $(p_n(D_T^2))_{n\in \mathbb{N}}$ converges to $D_T$, since $D_T$ is positive and $\|D_T\|\leq 1$. Consequently,
\[SD_T=\lim_{n\to \infty}Sp_n(D_T^2)=\lim_{n\to \infty} p_n(D_T^2)S=D_TS.\]
Also, we have $ S^*D_T=D_TS^*$. Since $S$ commutes with $T$, the pair $(T,A)$, where $ST=A$, is commuting. Now we have to follow And\^{o}'s explicit construction of isometric dilation for the commuting pair $(T,ST)$. Define isometries $W_T$ and $W_A$ on $\ell^2(\mathcal{H})$ by
\begin{align*}
W_T(x_0,x_1,x_2,x_3,\dots) = & (Tx_0,D_Tx_0,0,x_1,x_2, \dots)\\
W_A(x_0,x_1,x_2,x_3,\dots) = & (Ax_0,D_Ax_0,0,x_1,x_2, \dots)\\
= & (STx_0,D_Tx_0,0,x_1,x_2, \dots).
\end{align*}
Define an operator $\mathbf{G}$ on $\ell^2(\mathcal{H})$ by
\[
\mathbf{G}(x_0,x_1,x_2,x_3,x_4,x_5,x_6,x_7,x_8, \dots)=(x_0,G(x_1,x_2,x_3,x_4),G(x_5,x_6,x_7,x_8), \dots),
\]
where $G:\mathcal{H}\bigoplus \mathcal{H}\bigoplus \mathcal{H}\bigoplus \mathcal{H}\to \mathcal{H}\bigoplus \mathcal{H}\bigoplus \mathcal{H}\bigoplus \mathcal{H}$ is defined by
\[G(u_0,u_1,u_2,u_3)=(S^*u_0,u_1,u_2,u_3).\]
Then $\mathbf{G}$ is a unitary and
\[
\mathbf{G}^*(x_0,x_1,x_2,x_3,x_4,x_5,x_6,x_7,x_8, \dots)=(x_0,G^*(x_1,x_2,x_3,x_4),G^*(x_5,x_6,x_7,x_8), \dots),
\]
where $G^*:\mathcal{H}\bigoplus \mathcal{H}\bigoplus \mathcal{H}\bigoplus \mathcal{H}\to \mathcal{H}\bigoplus \mathcal{H}\bigoplus \mathcal{H}\bigoplus \mathcal{H}$ is given by
\[G^*(u_0,u_1,u_2,u_3)=(Su_0,u_1,u_2,u_3).\]
It is easy to see that $\mathbf{G}W_T$ and $W_A\mathbf{G}^*$ are isometries. We further show that $(\mathbf{G}W_T,W_A\mathbf{G}^*)$ is a commuting pair. 
\begin{align*}
& \mathbf{G}W_TW_A\mathbf{G}^* (x_0,x_1,x_2,x_3,x_4,x_5,x_6,x_7,x_8 \dots)\\
= &  \mathbf{G}W_TW_A(x_0,(Sx_1,x_2,x_3,x_4),(Sx_5,x_6,x_7,x_8), \dots)\\
= & \mathbf{G}W_T (Ax_0, D_Ax_0, 0, (Sx_1,x_2,x_3,x_4),(Sx_5,x_6,x_7,x_8), \dots)\\
= & \mathbf{G}W_T (STx_0, D_Tx_0, 0, (Sx_1,x_2,x_3,x_4),(Sx_5,x_6,x_7,x_8), \dots) \hspace{2 cm} (\mathrm{By~Lemma}~\ref{Simple Lemma})\\
= & \mathbf{G}(TSTx_0,D_TSTx_0,0, D_Tx_0, 0, (Sx_1,x_2,x_3,x_4),(Sx_5,x_6,x_7,x_8), \dots)\\
= & (TSTx_0,(S^*D_TSTx_0,0, D_Tx_0, 0), (x_1,x_2,x_3,x_4),(x_5,x_6,x_7,x_8), \dots)\\
= & (TSTx_0,(D_TS^*STx_0,0, D_Tx_0, 0), (x_1,x_2,x_3,x_4),(x_5,x_6,x_7,x_8), \dots) \hspace{0.4 cm} (\mathrm{since}~S^*D_T=D_TS^*)\\
= & (TSTx_0,(D_TTx_0,0, D_Tx_0, 0), (x_1,x_2,x_3,x_4),(x_5,x_6,x_7,x_8), \dots) \hspace{1.1 cm} (\mathrm{since}~SD_T=D_TS)\\
= & (STTx_0,(D_{ST}Tx_0,0, D_Tx_0, 0), (x_1,x_2,x_3,x_4),(x_5,x_6,x_7,x_8), \dots)\\
= & (ATx_0,D_ATx_0,0, D_Tx_0, 0, x_1,x_2,x_3,x_4,x_5,x_6,x_7,x_8, \dots)\\
= & W_A(Tx_0,D_Tx_0,0, x_1,x_2,x_3,x_4,x_5,x_6,x_7,x_8, \dots)\\
= & W_AW_T(x_0,x_1,x_2,x_3,x_4,x_5,x_6,x_7,x_8 \dots)\\
= & W_A\mathbf{G}^*\mathbf{G}W_T(x_0,x_1,x_2,x_3,x_4,x_5,x_6,x_7,x_8 \dots).
\end{align*}
Let $\widetilde{V_T}= \mathbf{G}W_T$ and $\widetilde{V_A}=W_A\mathbf{G}^*$. The pair $(\widetilde{V_T}, \widetilde{V_A})$ is a commuting pair of isometries. Observe that
\begin{align*}
&\widetilde{V_T}(x_0,x_1,x_2, \dots)=(Tx_0,S^*D_{T}x_0,0,x_1,x_2, \dots),  \qquad & (x_0,x_1,x_2,\dots)\in \ell^2(\mathcal{H}),\\
&\widetilde{V_A}(x_0,x_1,x_2, \dots)=(Ax_0,D_{A}x_0,0,Sx_1,x_2, \dots), \qquad & (x_0,x_1,x_2,\dots)\in \ell^2(\mathcal{H}).
\end{align*}
Therefore,
\[
\widetilde{V_T}^{n_1}\widetilde{V_A}^{n_2}(x_0,x_1,x_2, \dots)=(T^{n_1}A^{n_2}x_0,S^*D_{T}T^{n_1-1}A^{n_2}x_0,0,S^*D_{T}T^{n_1-2}A^{n_2}x_0, \dots), \quad  n_1,n_2 \in \mathbb N \cup \{0\}.
\]
Consequently,
\[
P_\mathcal{H}\widetilde{V_T}^{n_1}\widetilde{V_A}^{n_2}|_\mathcal{H}=T^{n_1}A^{n_2}, \quad  n_1,n_2 \in \mathbb N \cup \{0\}.
\]
Thus, $(\widetilde{V_T}, \widetilde{V_A})$ is an isometric dilation of $(T,A)$ on $\ell^2(\mathcal{H})$. We now show that $\widetilde{V_T}\perp_B \widetilde{V_A}$. For the unit vectors $x_k\,,$ let
\[
y_k=\left(\eta x_k,0,\zeta x_k,0,0,0,\dots\right), \qquad k\in \mathbb{N},
\]
where
\[
\eta=\frac{1}{\sqrt{1-\beta}}, \qquad \zeta=\sqrt{\frac{-\beta}{1-\beta}}.
\]
Then $(y_k)_{k\in \mathbb{N}}$ is a sequence of unit vectors in $\ell^2(\mathcal{H})$, and we have
\begin{align*}
& \left\langle \widetilde{V_T}y_k, \widetilde{V_A}y_k\right\rangle\\
= &\left\langle \widetilde{V_T}\left(\eta x_k,0,\zeta x_k,0,0,0,\dots\right), \widetilde{V_A}\left(\eta x_k,0,\zeta x_k,0,0,0,\dots\right)\right\rangle_{\ell^2(\mathcal{H})}\\
= & \left\langle \mathbf{G}W_T\left(\eta x_k,0,\zeta x_k,0,0,\dots\right), W_A\mathbf{G}^*\left(\eta x_k,0,\zeta x_k,0,0,0,\dots\right)\right\rangle_{\ell^2(\mathcal{H})}\\
= & \left\langle \mathbf{G}\left(\eta Tx_k,\eta D_Tx_k,0,0,\zeta x_k, 0,0,0,\dots\right), W_A\left(\eta x_k,0,\zeta x_k,0,0,0,\dots\right)\right\rangle_{\ell^2(\mathcal{H})}\\
= & \left\langle \left(\eta Tx_k,\eta S^* D_Tx_k,0,0,\zeta x_k, 0,0,0,\dots\right), \left(\eta STx_k,\eta D_Tx_k,0,0,\zeta x_k, 0,0,0, \dots\right)\right\rangle_{\ell^2(\mathcal{H})}\\
= & \eta^2\langle Tx_k, STx_k\rangle_\mathcal{H}+\eta^2\langle S^*D_Tx_k,D_Tx_k \rangle_\mathcal{H}+\zeta^2\|x_k\|^2\\
= & \eta^2\langle Tx_k, STx_k\rangle_\mathcal{H}+\eta^2\langle S^*D_T^2x_k,x_k \rangle_\mathcal{H}+\zeta^2\|x_k\|^2\qquad (\mathrm{since}~D_TS^*=S^*D_T)\\
= & \eta^2\langle Tx_k, STx_k\rangle_\mathcal{H}+\eta^2\langle S^*(I_\mathcal{H}-T^*T)x_k,x_k \rangle_\mathcal{H}+\zeta^2\|x_k\|^2\\
= & \eta^2\langle Tx_k, STx_k\rangle_\mathcal{H}+\eta^2\langle x_k,Sx_k\rangle_\mathcal{H}-\eta^2\langle T^*Tx_k,Sx_k \rangle_\mathcal{H}+\zeta^2\|x_k\|^2\\
= & \eta^2\langle x_k,Sx_k\rangle_\mathcal{H}+\zeta^2\|x_k\|^2
\end{align*}
Now,
\[
\lim_{k\to \infty} \eta^2\langle x_k,Sx_k\rangle_\mathcal{H}+\zeta^2=\eta^2\beta+\zeta^2=0.\]
Therefore, ${ \displaystyle \lim_{k\to \infty} \left\langle \widetilde{V_T}y_k, \widetilde{V_A}y_k\right\rangle=0}$ and $\widetilde{V_T}\perp_B \widetilde{V_A}$. It follows that (see CH-I of \cite{NFBK}) the commuting pair of isometries $\left(\widetilde{V_T}, \widetilde{V_A}\right)$ extends to a commuting pair of unitary operators $(U_T,U_A)$ on some Hilbert space $\mathcal{K}\supseteq \ell^2(\mathcal{H})$.
Therefore, it follows that (see CH-I of \cite{NFBK}) $(U_T,U_A)$ is commuting unitary dilation of $(T,A)$. Note that $U_T$ and $U_A$ are norm preserving extensions of $\widetilde{V_T}$ and $\widetilde{V_A}$, respectively. Consequently, by Lemma \ref{Orthogonality of extensions}, $U_T\perp_B U_A$, since $\widetilde{V_T}\perp_B \widetilde{V_A}$. This completes the proof.

\end{proof}

The construction in the proof of Theorem \ref{Class of examples} ensures that $\widetilde{U_T}\perp_B \widetilde{U_{ST}}$ and $\widetilde{V_T}\perp_B \widetilde{V_{ST}}$ when $S$ commutes with $T$, but $T$ and $ST$ may or may not be orthogonal. Indeed, the fact that $\widetilde{U_T}\perp_B \widetilde{U_{ST}}$ or $\widetilde{V_T}\perp_B \widetilde{V_{ST}}$ does not imply $T\perp_B {ST}$ even if $\|T\|=1$ and $S, T$ commute. The following example explains this.

\begin{example}
Let $T$ be a contraction acting on a Hilbert space $\mathcal{H}$ such that there is a sequence of unit vectors $(x_n)$ in $\mathcal{H}$ for which
\[
\lim_{n\to \infty }  2\|Tx_n\|^2 \geq 1.
\]
Choose $S=-I_\mathcal{H}$ and let $A=ST=-T$. Then by Theorem \ref{Class of examples}, we have $\widetilde{V_T}\perp_B \widetilde{V_A}$ as members of $\ell^2(\mathcal{H})$. Moreover, since
\[
\lim_{n\to \infty } (1-\|Tx_n\|^2+\langle Tx_n, STx_n\rangle)=\lim_{n\to \infty } (1-2\|Tx_n\|^2)\leq 0,
\]
by Theorem \ref{Class of examples}, we have $\widetilde{U_T}\perp_B \widetilde{U_A}$. However, it is evident that $T\not\perp_B A$ and $A\not\perp_B T$. \qed

\end{example}

We next turn our attention towards the orthogonality of regular unitary dilations. Let $\mathcal{T}=(T_1,T_2,\dots, T_k)$ be a commuting tuple of contractions on a Hilbert space $\mathcal{H}$. Given any $\alpha=(\alpha_1,\alpha_2,\dots,\alpha_k)\in \mathbb{Z}^k_+$, let $supp (\alpha) = \{i:~\alpha_i\neq 0\}$ and we define $\mathcal{T}^{\alpha}=T_1^{\alpha_1}T_2^{\alpha_2}\dots T_k^{\alpha_k}$. A commuting tuple of unitaries $\mathcal{U}=(U_1, U_2,\dots, U_k)$ on a Hilbert space $\mathcal{K}$ is said to be a \textit{regular dilation} of the tuple $\mathcal{T}$, if $\mathcal{H}\subseteq \mathcal{K}$ and
\[
(\mathcal{T}^{\alpha})^*\mathcal{T}^\beta = P_\mathcal{H}(\mathcal{U}^{\alpha})^*\mathcal{U}^\beta|_{\mathcal{H}},
\]
where $\alpha, \beta$  are members of $\mathbb{Z}^k_{+}$ with $supp~(\alpha) \cap supp~(\beta) = \emptyset$. A necessary and sufficient condition for $\mathcal{T}$ to have regular unitary dilation is that it satisfies Brehmer positivity condition (see\cite{Brehmer} or \cite[Section 9]{NFBK}) i.e.,
\[
\sum_{F\subseteq G}(-1)^{|F|}\mathbf{T}_F^*\mathbf{T}_F\geq \mathbf{0},
\]
for every $G\subseteq \{1,2, \dots, k\}$, where
\[
\mathbf{T}_F=\prod_{j\in F} T_j, \qquad \mathrm{and} \qquad \mathbf{T}_{\emptyset}=I_\mathcal{H}.
\]
We shall now see that if a pair of commuting contractions $(T_1,T_2)$ satisfies Brehmer's positivity condition then it admits a regular unitary dilation that preserves orthogonality. Indeed, if the maximal numerical range (as in Equation-(\ref{eqn:sec3-03})) $\mathcal W(T_2^*T_1)$ contains $0$, which is a weaker condition than $T_1\perp_B T_2$, then also $(T_1,T_2)$ possesses an orthogonal regular unitary dilation.
 
\begin{theorem}\label{Regular unitary dilation}
Let $\mathcal{T}=(T_1,T_2)$ be a commuting pair of contractions on a Hilbert space $\mathcal{H}$ that satisfies Brehmer's positivity condition. Suppose that $0\in \mathcal{W}(T_2^*T_1)$. Then the pair $(T_1,T_2)$ possesses regular a unitary dilation $\mathcal{U}=(U_1,U_2)$ on some Hilbert space $\mathcal{K}\supseteq \mathcal{H}$ such that $U_1\perp_B U_2$. In particular, if $T_1\perp_B T_2$, then $U_1\perp_B U_2$.

\end{theorem}

\begin{proof}
Since the system $\mathcal{T}$ satisfies Brehmer's positivity condition, it follows from \cite[Section 9]{NFBK} that the group $\mathbb{Z}_2$ possesses a unitary representation on a Hilbert space $\mathcal{K}\supseteq \mathcal{H}$. To be more precise, there exists a homomorphism $\phi:\mathbb{Z}_2\to \mathcal{B}({\mathcal{K}})$ such that 
\begin{itemize}
\item[$\bullet$] $\phi(g)$ is unitary for each $g\in \mathbb{Z}_2$.
\item[$\bullet$] $\phi(g^{-1})=\phi(g)^*$.
\item[$\bullet$] The system $(\phi(e_1), \phi(e_2))$ is a regular unitary dilation of $(T_1,T_2)$, where $e_1=(1,0)$ and $e_2=(0,1)$.

\end{itemize}
Since $0\in \mathcal{W}(T_2^*T_1)$, there exists a sequence of unit vectors $(x_n)$ such that
\[
\lim_{n\to \infty} \langle T_2^*T_1 x_n,x_n \rangle_\mathcal{H} = 0.
\]
Now,
\begin{align*}
\langle T_2^* T_1x_n, x_n \rangle_\mathcal{H}
= \langle P_\mathcal{H}\phi(e_2)^*\phi(e_1)x_n,x_n\rangle_\mathcal{K}
= \langle \phi(e_1)x_n,\phi(e_2)x_n\rangle_\mathcal{K}.
\end{align*}
Taking limit on the both sides we have
\[
\lim_{n\to\infty} \langle \phi(e_1)x_n,\phi(e_2)x_n\rangle_\mathcal{K}=0.
\]
Since every sequence of unit vectors is a norming sequence for a unitary, we have $\phi(e_1)\perp_B \phi(e_2)$. 

\medskip

It remains to show that if $T_1\perp_B T_2$, then $0 \in \mathcal W(T_2^*T_1)$. Indeed, if $T_1\perp_B T_2$, then there exists a norming sequence $(x_n)$ of $T_1$ such that 
$
\lim_{n\to\infty}\langle T_1 x_n, T_2 x_n \rangle=0.
$
So, we have $0\in \mathcal{W}(T_2^*T_1)$ and the proof is complete.

\end{proof}

For a pair of commuting contraction $(T_1,T_2)$ satisfying Brehmer's positivity condition, it may so happen that $0\in \mathcal{W}(T_2^*T_1)$ but $T_1\not\perp_B T_2$ and $T_2\not\perp_B T_1$ as the following example shows. 

\begin{example}
Consider the matrices 
\[
T_1=\begin{bmatrix}
1 & 0\\
0 & 0
\end{bmatrix} \quad \mathrm{and} \quad T_2=\begin{bmatrix}
1 & 0\\
0 & \frac{1}{2}
\end{bmatrix}.
\]
Evidently, $T_1T_2=T_2T_1$. Note that
\[
I_2-T_1^*T_1-T_2^*T_2+T_1^*T_2^*T_1T_2=\begin{bmatrix}
0 & 0\\
0 & \frac{3}{4}
\end{bmatrix}\geq 0.
\]
Therefore, the pair $(T_1,T_2)$ is a pair of commuting contraction that satisfies Brehmer positivity condition. It is easy to see that the norm attainment sets
\[M_{T_1}=M_{T_2}=\{\mu e_1:~|\mu|=1\},\]
and $\langle T_1e_1, T_2e_1\rangle = 1$. Thus, $T_1\not\perp_B T_2$. However, $0\in \mathcal{W}(T_2^*T_1)$, as $\langle T_2^*T_1 e_2, e_2 \rangle = 0.$
\end{example}
We conclude this article with the following corollary of Theorem \ref{Regular unitary dilation} whose proof is straight-forward.

\begin{cor}
Let $\mathcal{T}=(T_1,T_2)$ be a commuting pair of contractions on a Hilbert space $\mathcal{H}$. Then $(T_1,T_2)$ possesses an orthogonal regular unitary dilation $\mathcal{U}=(U_1,U_2)$ if anyone of the following holds.
\begin{itemize}
    \item[(i)] The pair $(T_1,T_2)$ is doubly commuting and $0\in \mathcal{W}(T_2^*T_1)$.
    \item[(ii)] The pair $(T_1,T_2)$ satisfies $\|T_1\|^2+\|T_2\|^2\leq 1$ and $0\in \mathcal{W}(T_2^*T_1)$.
   \end{itemize}
   \qed
\end{cor}

\vspace{0.4cm}

\noindent \textbf{Acknowledgement.} The first named author is supported by the ``Early Research Achiever Award Grant" of IIT Bombay with Grant No. RI/0220-10001427-001. The second named author is supported by the Institute Postdoctoral Fellowship of IIT Bombay.\\

\end{document}